\documentclass[a4paper,11pt]{article}
\usepackage{latexsym,amssymb}
\usepackage{amsmath}
\usepackage[mathscr]{eucal}
\usepackage{color}
\usepackage[abbrev]{amsrefs}

\usepackage{latexsym}
\usepackage{amsmath}
\usepackage{amssymb}
\usepackage{enumitem}
\usepackage{mathrsfs}
\usepackage{color}

\usepackage{mathtools}
\usepackage{amsthm}

\newtheorem{Theorem}{\bf Theorem}[section]
\newtheorem{Proposition}{\bf Proposition}[section]
\newtheorem{Lemma}{\bf Lemma}[section]

\theoremstyle{definition}
\newtheorem{Remark}{\bf Remark}[section]

\newtheorem{Definition}{\bf Definition}[section]

\newcommand{\jn}{\quad\textrm{in}\quad}

\newcommand{\for}{\quad\textrm{for}\quad}
\newcommand{\jf}{\quad\textrm{if}\quad}
\newcommand{\as}{\quad\textrm{as}\quad}

\newcommand{\with}{\quad\textrm{with}\quad}
\newcommand{\FA}{\quad\textrm{for any}\quad}
\newcommand{\FS}{\quad\textrm{for some}\quad}

\newcommand{\Ker}{\operatorname{Ker}}
\renewcommand{\Im}{\operatorname{Im}}
\newcommand{\dist}{\operatorname{dist}}

\newcommand{\supp}{\operatorname{supp}}
\newcommand{\loc}{{\mathrm{loc}}}
\newcommand{\R}{\mathbb{R}}
\newcommand{\integ}[1]{\int_{#1}}
\numberwithin{equation}{section}
\allowdisplaybreaks[1]

\title{Semilinear elliptic problems on the half space\\
with a supercritical nonlinearity}
\author{
Sho Katayama\footnote{Email: katayama-sho572@g.ecc.u-tokyo.ac.jp}\vspace{5pt}\\
Graduate School of Mathematical Sciences,\\
The University of Tokyo,\\
3-8-1 Komaba, Meguro-ku, Tokyo 153-8914, Japan
}
\date{}
\begin{document}
\maketitle
\renewcommand{\thefootnote}{\fnsymbol{footnote}}
\footnote[0]{2020AMS Subject Classification: 35J25, 35J61, 35B09, 35B32}
\footnote[0]{Keywords: scalar field equation, supercritical, multiple positive solutions, inhomogeneous boundary condition, Joseph--Lundgren exponent}
\renewcommand{\thefootnote}{\arabic{footnote}}
\begin{abstract}
This paper concerns positive solutions to the boundary value problems of the scalar field equation in the half space with a Sobolev supercritical nonlinearity and an inhomogeneous Dirichlet boundary condition, admitting a nontrivial nonnegative Radon measure as the boundary data. Under a suitable integrability assumption on the boundary data and the Joseph--Lundgren subcritical condition on the nonlinear term, we give a complete classification of the existence/nonexistence of a positive solution with respect to the size of the boundary data. Furthermore, we give a result on the existence of multiple positive solutions via bifurcation theory.
\end{abstract}
\section{Introduction}
This paper concerns the existence/nonexistence and multiplicity of solutions to the boundary value problem for a semilinear scalar field equation
\begin{equation}
\tag{$\mbox{P}_\kappa$}\label{P}
\left\{\begin{aligned}
-\Delta u+u&=u^p&\jn&\R^N_+,\\
u&>0&\jn&\R^N_+,\\
u(x',x_N)&\to 0&\as& x_N\to+\infty\quad\for x'\in\R^{N-1},\\
u(x',0)&=\kappa\mu(x')&\for& x'\in\R^{N-1}.
\end{aligned}\right.
\end{equation}
Here $N\ge 1$, $\R^N_+$ is the half space $\{(x_1,\ldots,x_N)\in\R^N: x_N>0\}$, $p>1$, $\kappa>0$, and $\mu$ is a nontrivial nonnegative Radon measure on $\R^{N-1}$. We are especially interested in the Sobolev supercritical case $p>p_S$, where
\[
p_S:=\infty\jf N\le 2,\qquad p_S:=\frac{N+2}{N-2}\jf N\ge 3.
\]
In general, one of difficulties of elliptic problems with supercritical nonlinearities is that it is difficult to treat them as variational problems for energy functionals on Sobolev spaces, due to the failure of the embedding $H^1_0\subset L^{p+1}$. For this reason, the existence/nonexistence and multiplicity of solutions to elliptic problems with supercritical nonlinearities are widely open. We also notify that solutions to elliptic equations with supercritical nonlinearities have quite a little different structures from those to elliptic equations with subcritical or critical nonlinearities.

The aim of this paper is to classify the existence/nonexistence and multiplicity of positive solutions to problem~\eqref{P} with respect to the value of parameter $\kappa$, under a wide assumption on the boundary data $\mu$. In particular, we prove the existence of a threshold constant $\kappa^*\in(0,\infty)$ with the following properties.
\begin{enumerate}[label={\rm(\Alph*)}]
\item If $0<\kappa<\kappa^*$, then problem~\eqref{P} possesses a solution.
\item If $\kappa>\kappa^*$, then problem~\eqref{P} possesses no solutions.
\item If $1<p<p_{JL}$ and $\kappa=\kappa^*$, then problem~\eqref{P} possesses a unique solution.
\item If $1<p<p_{JL}$, then there is a constant $\kappa_*\in[0,\kappa^*)$ such that if $\kappa_*<\kappa<\kappa^*$, then problem~\eqref{P} possesses at least two solutions.
\end{enumerate}
Here $p_{JL}$ is the so-called Joseph--Lundgren critical exponent, that is,
\[
p_{JL}:=\infty\jf N\le 10,\quad p_{JL}:=\frac{N^2-8N+4+8\sqrt{N-1}}{(N-2)(N-10)}\jf N\ge 11.
\]
Note that $p_{JL}>p_S$ if $N\ge 3$. $p_{JL}$ appears in various aspects of nonlinear elliptic equations, such as bifurcations and stability of solutions to semilinear elliptic problems (see e.g. \cites{CR1,CR2,F,IOS01,IOS03,IK,JL} and references therein).

There are many studies concerning the existence/nonexistence and multiplicity of positive solutions to semilinear elliptic equations (see e.g. \cites{JunAi, BaL, BL1, BL2, BGNV, BV, CR1, CR2, EL, FW, FIK1, FIK2, Hsu1, HL1, HL2, IK, IOS01, IOS03, JL, LTW, NS01}). For instance, Ai and Zhu \cite{JunAi} proved the existence of $\kappa^*>0$ with properties (A)-(D) in the sense of weak solutions in $H^1(\R^N_+)$ with $\kappa_*=0$, under the assumptions $N\ge 2$, $1<p<p_S$, and $\mu\in H^{1/2}(\R^{N-1})\cap L^\infty(\R^{N-1})$ (see Remark~\ref{muassump}). 

We introduce some notations and formulate the definition of solutions to problem~\eqref{P}. For $x=(x_1,\ldots,x_N)\in\R^N$, we denote by $x^*$ the reflected point of $x$ with respect to the plane $\partial\R^N_+=\{x_N=0\}$, that is,
\[
x^*:=(x_1,\ldots,x_{N-1},-x_N).
\]
Let $E=E(|x|)$ be the fundamental solution for the scalar field operator $-\Delta+1$~on~$\R^N$, that is,
\[
E(r):=(2\pi)^{-N/2}r^{(2-N)/2}K_{(N-2)/2}(r)\for r>0,
\]
where $K_{(N-2)/2}$ is the modified Bessel function of order $(N-2)/2$. For $r>0$ and $x\in\R^N$, we write
\[
B(x,r):=\{y\in\R^N: |y-x|<r\}.
\]
We define the Dirichlet--Green kernel $G$ and the Poisson kernel $P$ for the operator $-\Delta+1$~on~$\R^N_+$ by
\[
G(x,y):=E(|x-y|)-E(|x^*-y|)\for x,y\in\overline{\R^N_+}\quad\mathrm{with}\quad x\neq y,
\]
\[
P(x,z):=\left.\frac{\partial}{\partial s}G(x,(z,s))\right|_{s=0}\for x\in\R^N_+,\quad z\in\R^{N-1}.
\]
Since $E$ is strictly decreasing, $G(x,y)>0$ and $P(x,z)>0$, for any $x,y\in\R^N_+$ and $z\in\R^{N-1}$. We remark here that the same argument as in \cite{GT}*{Section 2.4} derives
\begin{equation}\label{GRepr}
u(x)=\integ{\R^{N-1}}P(x,z)u(z)dz+\integ{\R^N_+}G(x,y)(-\Delta u(y)+u(y))dy
\end{equation}
for any $u\in C^2(\overline{\R^N_+})\cap W^{2,\infty}(\R^N_+)$. We also note that the integral kernel $G(x,y)$ is symmetric, that is,
\[
G(x,y)=G(y,x)\FA x,y\in\overline{\R^N_+}\with x\neq y.
\]
For simplicity, we write
\[
P[\mu](x):=\integ{\R^{N-1}}P(x,z)d\mu(z),\quad G[f](x):=\integ{\R^N_+}G(x,y)f(y)dy,
\]
for a Radon measure $\mu$ on $\R^{N-1}$ and a measurable function $f$ on $\R^{N}_+$. 

Also,
 we introduce a weighted Lebesgue space $L^q_\alpha$. For $q\in[1,\infty)$ and $\alpha\in\R$, we define the space $L^q_\alpha$ by
\[
L^q_\alpha:=\left\{f\in L^q_{\mathrm{loc}}(\R^N_+):h(x_N)^\alpha f\in L^q(\R^N_+)\right\},\quad h(t):=\begin{cases}
t&\for 0<t<1,\\
1&\for t\ge 1,
\end{cases}
\]
with the norm
\[
\|f\|_{L^q_\alpha}:=\left(\integ{\R^N_+}|f(x)|^qh(x_N)^{q\alpha}dx\right)^{1/q}.
\]
\begin{Definition}
Let $q\ge p>1$, $\alpha\ge 0$, $\kappa>0$, and $\mu$ be a nonnegative Radon measure on $\R^{N-1}$.
\begin{enumerate}[label={\rm (\roman*)}]
\item We say that $u\in L^q_\alpha$ is an $L^q_\alpha$-solution to problem~\eqref{P} if
\begin{equation}\label{IE}
u(x)=\kappa P[\mu](x)+G[u^p](x)>0
\end{equation}
for almost all (a.a.) $x\in\R^{N}_+$.
\item We say that $v\in L^q_\alpha$ is an $L^q_\alpha$-supersolution to problem~\eqref{P} if
\[
v(x)\ge\kappa P[\mu](x)+G[v^p](x)>0
\]
for a.a. $x\in\R^{N}_+$.
\item We say that an $L^q_\alpha$-solution $u$ to problem~\eqref{P} is a minimal $L^q_\alpha$-solution to problem~\eqref{P} if any $L^q_\alpha$-solution $\tilde{u}$ to problem~\eqref{P} satisfies $u(x)\le\tilde{u}(x)$ for a.a. $x\in\R^N_+$.
\end{enumerate}
\end{Definition}
Note that a minimal $L^q_\alpha$-solution to problem \eqref{P} is unique if it exists.

Now we are ready to state our results. Our first theorem concerns properties (A) and (B), that is, the existence of a solution to problem~\eqref{P} with $\kappa$ small and the nonexistence of a solution to problem~\eqref{P} with $\kappa$ large.
\begin{Theorem}\label{Thm1}
Let $1<p<q<\infty$ and $\alpha\ge 0$ be such that
\begin{equation}\label{qalpha}
\frac{1}{q}+\alpha<\frac{2}{p},\quad \frac{N}{q}+\alpha<\frac{2}{p-1}.
\end{equation}
Let $\mu$ be a nontrivial nonnegative Radon measure on $\R^{N-1}$ and assume that
\begin{equation}\label{pmu}
P[\mu]\in L^q_\alpha.
\end{equation}
Then there is $\kappa^*>0$ with the following properties.
\begin{enumerate}[label={\rm (\roman*)}]
\item If $0<\kappa<\kappa^*$, then problem~\eqref{P} possesses a minimal $L^q_{\alpha}$-solution $u^\kappa$.
\item If $\kappa>\kappa^*$, then problem~\eqref{P} possesses no $L^q_\alpha$-solutions.
\end{enumerate}
\end{Theorem}
Our second theorem, which is the main result of this paper, gives a complete classification of the existence/nonexistence of solutions to problem~\eqref{P} under the condition $1<p<p_{JL}$. Furthermore, it gives the existence of multiple solutions to problem~\eqref{P} with $\kappa$ slightly less than $\kappa^*$.
\begin{Theorem}\label{Thm2}
Assume the same conditions as in Theorem~{\rm\ref{Thm1}} and $1<p<p_{JL}$. Let $\kappa^*$ be as in Theorem~{\rm\ref{Thm1}}.
\begin{enumerate}[label={\rm (\roman*)}]
\item If $0<\kappa\le\kappa^*$, then problem~\eqref{P} possesses a minimal $L^q_\alpha$-solution $u^{\kappa}$. Furthermore, if $\kappa=\kappa^*$, then $u^{\kappa^*}$ is the unique $L^q_\alpha$-solution to problem~\eqref{P}.
\item If $\kappa>\kappa^*$, then problem~\eqref{P} possesses no solutions.
\item There is a constant $0\le\kappa_*<\kappa^*$ such that if $\kappa_*<\kappa<\kappa^*$, then problem~\eqref{P} possesses at least two $L^q_\alpha$-solutions.
\end{enumerate}
\end{Theorem}
\begin{Remark}
\begin{enumerate}[label={\rm(\roman*)}]
\item It follows from the condition \eqref{qalpha} and $\alpha\ge 0$ that
\begin{equation}\label{qNp-12}
q>\frac{N}{2}(p-1),
\end{equation}
which together with elliptic regularity theorems implies that an $L^q_\alpha$-solution $u$ to problem~\eqref{P} belongs to $C^{2,\gamma}(\R^N_+)$ for any $\gamma\in(0,1)$, and it satisfies
\[
-\Delta u+u=u^p\jn\R^N_+,\quad u(x',x_N)\to 0\as x_N\to+\infty\for x'\in\R^{N-1}.
\]
\item Even in the subcritical case $1<p<p_S$, our result is new in the sense that it admits some singular boundary data. We will discuss on the assumption $P[\mu]\in L^q_\alpha$ in \ref{pmucond}.
\end{enumerate}
\end{Remark}

The proof of Theorem~\ref{Thm1} is based on the contraction mapping theorem on $L^q_\alpha$, via the integral inequality for the integral kernel $G$ on the spaces $L^q_\alpha$ (see Proposition~\ref{Glaa} and \ref{pLaG}).

In the proof of Theorem~\ref{Thm2}, the ``stability'' of the minimal $L^q_\alpha$-solution $u^\kappa$, that is,
\[
\integ{\R^N_+}p(u^\kappa)^{p-1}\omega^2dx\le\integ{\R^N_+}(|\nabla\omega|^2+\omega^2)dx\FA\omega\in W^{1,2}_0(\R^N_+),\quad \kappa\in(0,\kappa^*),
\]
plays a crucial role. The proof of the existence of a solution to problem~\eqref{P} with $\kappa=\kappa^*$ is based on uniform estimates of the minimal $L^q_\alpha$-solutions $\{u^\kappa\}_{\kappa\in(0,\kappa^*)}$, using the ``stability'' of $u^\kappa$ and elliptic regularity theorems. The proof of assertion (iii) is based on bifurcation theory.
\begin{Remark}
Due to the unboundedness of the half space, it is not clear whether the linearized operator $h\mapsto h-G[p(u^\kappa)^{p-1}h]$ is Fredholm of index zero, which is an essential condition for the use of bifurcation theory. Although bifurcations of solutions to elliptic problems on unbounded domains were discussed in e.g. \cite{IOS03}, we provide a new method to recover the bifurcation theory, which is simpler and applies to more general settings than previous studies.
\end{Remark}

The rest of the paper is organized as follows. In Section~\ref{LaG}, we summarize properties of the integral kernel $G$ and the function space $L^q_\alpha$. In Section~\ref{smallk}, we prove the existence of a minimal $L^q_\alpha$-solution for small $\kappa$. In Section~\ref{weaksolsec}, we reduce problem~\eqref{P} into a weak form. In Section~\ref{EVPsec}, we study the linearized eigenvalue problem to obtain the ``stability'' of the minimal $L^q_\alpha$-solutions $u^\kappa$, and complete the proof of Theorem~\ref{Thm1}. In Section~\ref{unformest}, we obtain uniform estimates of the minimal $L^q_\alpha$-solutions $u^\kappa$, and prove the existence of an $L^q_\alpha$-solution to problem~\eqref{P} with $\kappa=\kappa^*$ under the condition $1<p<p_{JL}$. In Section~\ref{PfThm2}, we complete the proof of Theorem~\ref{Thm2}.

\section{Properties of the integral kernel $G$ and the function space $L^q_\alpha$}\label{LaG}
We first state some properties of the integral kernel $G$ and the function spaces $L^q_\alpha$, which play key roles in the proofs of main theorems. In what follows, we use $C$ to denote generic positive constants and point out that $C$ may take different values within a calculation.
\subsection{The integral inequality of the integral kernel $G$ on the spaces $L^q_\alpha$}
We first give the following integral inequality, which describes a fundamental property of the integral kernel $G$ on the spaces $L^q_\alpha$. The proof of Proposition~\ref{Glaa} is given in \ref{pLaG}.
\begin{Proposition}\label{Glaa}
Let $1<q\le r<\infty$ and $\alpha,\beta\in\R$ be such that
\[
\frac{1}{q}+\alpha<2,\quad \frac{1}{r}+\beta>-1,\quad\frac{1}{q}-\frac{1}{r}<\frac{2}{N},\quad \frac{N}{r}+\beta\ge \frac{N}{q}+\alpha-2.
\]
Then for any $f\in L^q_\alpha$,
\begin{equation}\label{Gla}
G[f]\in L^r_\beta,\quad \|G[f]\|_{L^r_\beta}\le C\|f\|_{L^q_\alpha}.
\end{equation}
\end{Proposition}
\begin{Remark}
The assumption $1/q+\alpha<2$ is sharp. Indeed, let $\sigma\in(1/q,1)$ and
\[
f(x):=x_N^{-2}\left(\log\frac{1}{x_N}\right)^{-\sigma}\jf |x|<\frac{1}{2},\quad
f(x):=0\jf |x|\ge\frac{1}{2}.
\]
Then one can easily see that $f\in L^q_{2-1/q}$. But since $\integ{\R^N_+}x_Nf(x)dx=\infty$, $G[f]\equiv\infty$.

The assumption $1/r+\beta>-1$ is also sharp. Indeed, since the dual space of $L^q_\alpha$ is $L^{q/(q-1)}_{-\alpha}$ and $G$ is symmetric, \eqref{Gla} implies that
\begin{equation}\label{dual}
G[g]\in L^{q/(q-1)}_{-\alpha},\quad \|G[g]\|_{L^{q/(q-1)}_{-\alpha}}\le C\|g\|_{L^{r/(r-1)}_{-\beta}},\FA g\in L^{r/(r-1)}_{-\beta}
\end{equation}
by a duality argument. This together with the argument above implies that $1-1/r-\beta<2$, which is equivalent to $1/r+\beta>-1$, is a necessary condition for \eqref{Gla}.
\end{Remark}
\begin{Remark}
We mention some related known integral inequalities. Dou and Ma \cite{DM}*{Theorem 1.15} proved that under conditions
\[
1<q\le r<\infty,\quad\alpha\ge\beta,\quad\frac{1}{q}+\alpha<1,\quad \frac{1}{r}+\beta>0,\quad \frac{N}{r}+\beta=\frac{N}{q}+\alpha-2,
\]
the solution $v$ to elliptic problem $-\Delta v=f$~in~$\R^N_+$,\quad $v=0$~on~$\partial\R^N_+$ satisfies $\|v\|_{L^r_\beta}\le C\|f\|_{L^q_\alpha}$.

Integral inequalities for the Dirichlet--Green functions on weighted Lebesgue spaces on bounded domains have been studied. Dur\'{a}n, Sanmartino, and Toschi \cite{DST} proved that under conditions
\[
\alpha\ge 0,\quad\frac{1}{q}+\alpha<1,\quad q\alpha=r\beta,\quad \frac{1}{q}-\frac{1}{r}\le\frac{2}{N+\lceil q\alpha\rceil},
\]
the solution $v$ to elliptic problem $-\Delta v=f$~in~$\Omega$, $v=0$~on~$\partial\Omega$ satisfies $\|v\|_{L^r_{\beta}(\Omega)}\le C\|f\|_{L^q_\alpha(\Omega)}$. Here the norm $\|\cdot\|_{L^q_\alpha(\Omega)}$ is defined by $\|f\|_{L^q_\alpha(\Omega)}:=\|\dist(\cdot,\partial\Omega)^\alpha f\|_{L^q(\Omega)}$. Their proof was based on the Muckenhoupt class $A_q$. The same inequality under conditions $q\alpha=r\beta=1$ and $1/q-1/r<2/(N+1)$ was proved by Fila, Souplet, and Weissler \cite{FSW} with a different method (see also \cite[Theorem 49.2]{QS}).
\end{Remark}
\subsection{Sobolev embeddings}
We next state Sobolev embeddings of the $L^q_\alpha$ spaces. Since these are easy consequences of interpolations between the Hardy inequality and the Sobolev inequality, we omit the proof.
\begin{Proposition}\label{Lqa}
\begin{enumerate}[label={\rm(\roman*)}]
\item Let $\alpha\in[0,1]$ and $r\in[2,\infty)$ be such that
\[
\frac{1}{r}\ge\frac{1}{2}-\frac{1-\alpha}{N}\jf N\neq 2,\quad r<\frac{2}{\alpha}\jf N=2.
\]
Then $W^{1,2}_0(\R^N_+)\subset L^r_{-\alpha}$ and
\begin{equation}\label{W1toLqa}
\|f\|_{L^r_{-\alpha}}\le C\|f\|_{W^{1,2}(\R^N_+)}\quad\textrm{for any}\quad f\in W^{1,2}_0(\R^N_+).
\end{equation}
\item Let $\alpha\in[0,1]$ and $q\in[1,2]$ be such that
\[
\frac{1}{q}\le\frac{1}{2}+\frac{1-\alpha}{N}\jf N\neq 2,\quad
\frac{1}{q}<1-\frac{\alpha}{2}\jf N=2.
\]
Then $L^q_\alpha\subset W^{-1,2}(\R^N_+)$ and
\[
\|f\|_{W^{-1,2}(\R^N_+)}\le C\|f\|_{L^q_\alpha}\quad\textrm{for any}\quad f\in L^q_\alpha.
\]
\end{enumerate}
\end{Proposition}
\begin{Remark}\label{120to-12}
It is an immediate consequence of Proposition~\ref{Lqa} that
\[
\|cv\|_{W^{-1,2}(\R^N_+)}\le C\|cv\|_{L^{q}_{\alpha}}\le C\|v\|_{L^r_{-\alpha}}\le C\|v\|_{W^{1,2}_0(\R^N_+)}
\]
for any $v\in W^{1,2}_0(\R^N_+)$ and $c\in L^s_\tau$ with $N/s<2-\tau$, $\tau\in[0,2)$. Here $1/q=1/2+(1-\alpha)/N$, $1/r=1/2-(1-\alpha)/N$, $\alpha=\tau/2$.
\end{Remark}
\begin{Remark}\label{120}
Assertion (ii) is mainly used to ``translate'' the sense of solutions to problem $-\Delta v+v=f$~in~$\R^N_+$, $v=0$~on~$\partial\R^N_+$, from the sense $v=G[f]$ into the weak sense in $W^{1,2}_0(\R^N_+)$, or conversely. Indeed, the following fact is well-known.
\begin{Proposition}\label{G120}
For any measurable function $f$ on $\R^N_+$ with $f,|f|\in W^{-1,2}(\R^N_+)$, $v:=G[f]$ is the unique weak solution to equation $-\Delta v+v=f$~in~$\R^N_+$, $v\in W^{1,2}_0(\R^N_+)$.
\end{Proposition}
\end{Remark}
\begin{Remark}\label{sigmaRmk}
Let $\alpha\in[0,1)$ and $r\in(2,\infty)$ be such that
\[
\frac{1}{r}\in\left(\frac{1}{2}-\frac{1-\alpha}{N},\frac{1}{2}\right).
\]
Then one can easily see from assertion (i) and the H\"{o}lder inequality that 
\[
\|v\|_{L^r_{-\alpha}}\le C\|v\|_{L^2(\R^N_+)}^{\sigma}\|v\|^{1-\sigma}_{W^{1,2}(\R^N_+)}\FA v\in W^{1,2}_0(\R^N_+)
\]
for small enough $\sigma\in(0,1)$. This implies that
\[
\integ{\R^N_+}cv^2dx\le C\|v\|_{L^2(\R^N_+)}^{2\sigma}\|v\|^{2(1-\sigma)}_{W^{1,2}(\R^N_+)}\FA v\in W^{1,2}_0(\R^N_+)
\]
for any $c\in L^s_\tau$ with $N/s<2-\tau$, $\tau\in[0,2)$. Thus, standard $L^2$-framework arguments for linear elliptic equations (see e.g. \cite[Chapter 8]{GT}) is applicable to solutions $v\in W^{1,2}_0(\R^N_+)$ to equation
\[
-\Delta v+cv=f\jn\R^N_+
\]
with $c,f\in L^s_\tau$, $N/s<2-\tau$, $\tau\in[0,2)$. More generally, this argument is valid for $c,f\in L^{s_1}_{\tau_1}+\ldots+L^{s_k}_{\tau_k}$ with $N/s_i<2-\tau_i$, $\tau\in[0,2)$, for each $i\in\{1,\ldots,k\}$.
\end{Remark}
\subsection{The compactness of operator $f\mapsto G[af]$}
We state the following compactness theorem for the integral operator $f\mapsto G[af]$. This theorem plays a key role in the proof of Theorem~\ref{Thm2}. Specifically, it verifies that the linearized operator $h\mapsto h-G[p(u^\kappa)^{p-1}h]$ is Fredholm of index zero, and enables us to use bifurcation theory to obtain the multiplicity assertion (iii). 
The proof of Proposition~\ref{cpt} is given in \ref{pLaG}.
\begin{Proposition}\label{cpt}
Let $r\in(1,\infty)$, $\beta\in\R$, $s'\in(N/2,\infty)$, and $\theta\in\R$ be such that
\begin{equation}\label{Cptassump}
\frac{1}{r}+\frac{1}{s'}< 1,\quad \frac{1}{r}+\frac{1}{s'}+\beta+\theta<2,\quad\frac{1}{r}+\beta>-1,\quad \frac{N}{s'}+\theta\le 2. 
\end{equation}
Let $a\in L^{s'}_{\theta}$ and set $T_af:=G[af]$ for a measurable function $f$ on $\R^N_+$. Then $T_a$ is a compact operator from $L^r_\beta$ to itself.
\end{Proposition}

\section{Existence of an $L^q_\alpha$-solution for small $\kappa$}\label{smallk}
In what follows, we always assume the same conditions as in Theorem~\ref{Thm1}. Let
\[
\mathcal{K}:=\left\{\kappa>0:\textrm{problem \eqref{P} possesses an $L^q_\alpha$-solution}\right\},\quad\kappa^*:=\sup\mathcal{K}.
\]
In this section, we prove the following proposition.
\begin{Proposition}\label{Kform}
Assume the same conditions as in Theorem~{\rm\ref{Thm1}}. Then the set $\mathcal{K}$ equals to either $(0,\kappa^*]$ or $(0,\kappa^*)$. Furthermore, $\kappa^*\in(0,\infty]$.
\end{Proposition}
The proof of Proposition~\ref{Kform} is based on the supersolution method and the contraction mapping theorem via Proposition~\ref{Glaa}. We first ensure that the supersolution method is available.
\begin{Lemma}\label{L3.1}
Assume the same conditions as in Theorem~{\rm\ref{Thm1}}. Then for any $\kappa>0$, the following assertions are equivalent.
\begin{enumerate}[label={\rm(\roman*)}]
\item $\kappa\in\mathcal{K}$.
\item Problem~\eqref{P} possesses a minimal $L^q_\alpha$-solution.
\item Problem~\eqref{P} possesses an $L^q_\alpha$-supersolution.
\end{enumerate}
Furthermore, if $u$ is a minimal $L^q_\alpha$-solution to problem~\eqref{P} and $v$ is an $L^q_\alpha$-supersolution to problem~\eqref{P}, then $u\le v$ a.e. in $\R^N_+$.
\end{Lemma}
\begin{proof}
We first prove the equivalence of (i), (ii), and (iii). Since it is clear that (ii)$\implies$(i)$\implies$(iii), it suffices to show that (iii)$\implies$(ii). We define approximate solutions $\{U^\kappa_j\}_{j\in\{0,1,\ldots\}}$ to problem~\eqref{P} by
\[
U^\kappa_0:=P[\mu],\quad U^\kappa_{j+1}:=\kappa P[\mu]+G[(U^\kappa_j)^p]\for j\in\{0,1,\ldots\}.
\]
Let $\kappa$ satisfy (iii). For any $L^q_\alpha$-supersolution $v$ to problem~\eqref{P}, it follows inductively from the definition of an $L^q_\alpha$-supersolution that
\[
U^\kappa_0\le U^\kappa_1\le\ldots\le v\quad\textrm{a.e. in}\quad \R^N_+.
\]
Thus the limit
\[
u^\kappa(x):=\lim_{j\to\infty}U^\kappa_j(x)
\]
exists and satisfies $u^\kappa(x)\le v(x)$ for a.a. $x\in\R^N_+$. This implies that $u^\kappa\in L^q_\alpha$ and that $u^\kappa\le u$ a.e. in $\R^N_+$ for any $L^q_\alpha$-solution $u$ to problem \eqref{P}. Furthermore, it follows from the monotone convergence theorem that $u^\kappa=\kappa P[\mu]+G[(u^\kappa)^p$] a.e. in $\R^N_+$, which proves that (iii)$\implies$ (ii). This proof also implies that if $u$ is a minimal $L^q_\alpha$-solution to problem~\eqref{P} and $v$ is an $L^q_\alpha$-supersolution to problem~\eqref{P}, then $u\le v$ a.e. in $\R^N_+$, and the proof of Lemma~\ref{L3.1} is complete.
\end{proof}
In what follows, let $u^\kappa$ denote the minimal $L^q_\alpha$-solution to problem~\eqref{P} for $\kappa\in\mathcal{K}$.
\begin{Lemma}\label{existence}
Assume the same conditions as in Theorem~{\rm\ref{Thm1}}. Then the following properties hold.
\begin{enumerate}[label={\rm(\roman*)}]
\item $(0,\kappa]\subset\mathcal{K}$ for any $\kappa\in\mathcal{K}$.
\item $u^\kappa<u^{\kappa'}$ a.e. in $\R^N_+$ for any $\kappa,\kappa'\in\mathcal{K}$ with $\kappa<\kappa'$.
\end{enumerate}
\end{Lemma}
\begin{proof}
Let $0<\kappa<\kappa'$, $\kappa'\in\mathcal{K}$, and $v:=u^{\kappa'}-(\kappa'-\kappa)P[\mu]$. Since
\[
v(x)=G[(u^{\kappa'})^p](x)+\kappa' P[\mu]-(\kappa'-\kappa)P[\mu](x)\ge\kappa P[\mu](x)>0,
\]
\begin{align*}
\kappa P[\mu](x)+G[v^p](x)&\le (\kappa'P[\mu](x)+G[(u^{\kappa'})^p](x))-(\kappa'-\kappa)P[\mu](x)\\
&=u^{\kappa'}(x)-(\kappa'-\kappa)P[\mu](x)=v(x),
\end{align*}
for a.a. $x\in\R^N_+$. $v$ is an $L^q_\alpha$-supersolution to problem~\eqref{P}. This together with Lemma~\ref{L3.1} implies $\kappa\in\mathcal{K}$ and $u^\kappa\le v<u^{\kappa'}$ a.e. in $\R^N_+$, which completes the proof of Lemma~\ref{existence}.
\end{proof}
Now we prove that $\kappa^*>0$ and complete the proof of Proposition~\ref{Kform}.
\begin{proof}[Proof of Proposition {\rm\ref{Kform}}]
It is easy to see that assertion (i) of Lemma~\ref{existence} implies that the set $\mathcal{K}$ coincides with either $(0,\kappa^*]$ or $(0,\kappa^*)$. It remains to prove that $\kappa^*>0$.

For $v\in L^q_\alpha$ and $\kappa>0$, we set
\begin{equation}\label{Psi}
\Psi(v,\kappa):=\kappa P[\mu]+G[v_+^p].
\end{equation}
Since
\begin{gather*}
\frac{p}{q}+p\alpha=p\left(\frac{1}{q}+\alpha\right)<2,\quad\frac{1}{q}+\alpha>0>-1,\\
\frac{p}{q}-\frac{1}{q}=\frac{p-1}{q}<\frac{2}{N},\\
\frac{N}{q}+\alpha=\frac{Np}{q}+p\alpha-(p-1)\left(\frac{N}{q}+\alpha\right)>\frac{Np}{q}+p\alpha-2,
\end{gather*}
by \eqref{qalpha} and \eqref{qNp-12}, it follows from Proposition  \ref{Glaa} that
\begin{equation}\label{Lq/p1}
\|G[w]\|_{L^q_\alpha}\le C\|w\|_{L^{q/p}_{p\alpha}}\FA w\in L^{q/p}_{p\alpha}.
\end{equation}
Let $v,v_1,v_2\in L^q_\alpha$. It follows from \eqref{Lq/p1} that
\begin{gather}
\Psi(v,\kappa)\in L^q_\alpha,\label{Phiest0}\\
\|\Psi(v,\kappa)\|_{L^q_\alpha}\le C\left(\kappa\|P[\mu]\|_{L^q_\alpha}+\|v_+^p\|_{L^{q/p}_{p\alpha}}\right)\le C\left(\kappa+\|v\|_{L^q_\alpha}^p\right),\label{Phiest1}\\
\begin{multlined}
\|\Psi(v_1,\kappa)-\Psi(v_2,\kappa)\|_{L^q_\alpha}=\|G[(v_1)_+^p-(v_2)_+^p]\|_{L^q_\alpha}\le C\|(v_1)_+^p-(v_2)_+^p\|_{L^{q/p}_{p\alpha}}\\
\le C\|\max\{(v_1)_+,(v_2)_+\}^{p-1}|v_1-v_2|\|_{L^{q/p}_{p\alpha}}\le C\left(\max\left\{\|v_1\|_{L^q_\alpha},\|v_2\|_{L^q_\alpha}\right\}\right)^{p-1}\|v_1-v_2\|_{L^q_\alpha}.
\end{multlined}\label{Phiest2}
\end{gather}
Let $\delta,\varepsilon>0$ and $\mathcal{F}_\delta:=\{v\in L^q_\alpha:\|v\|_{L^q_\alpha}\le\delta\}$. Set $\kappa=\varepsilon\delta$. Then \eqref{Phiest1} and \eqref{Phiest2} imply that
\[
\|\Psi(v,\kappa)\|_{L^q_\alpha}\le\delta,\quad\|\Psi(v_1,\kappa)-\Psi(v_2,\kappa)\|_{L^q_\alpha}\le\frac{1}{2}\|v_1-v_2\|_{_{L^q_\alpha}},
\]
for $v,v_1,v_2\in\mathcal{F}_\delta$ with $\delta$ and $\varepsilon$ sufficiently small. Using the contraction mapping theorem on $\Psi(\cdot,\kappa)$ on $\mathcal{F}_\delta$, we see that there is a function $u\in \mathcal{F}_\delta$ such that
\[
u=\Psi(u,\kappa)=\kappa P[\mu]+G[(u_+)^p].
\]
Since $u>0$~in~$\R^N_+$, $u$ is an $L^q_\alpha$-solution to problem~\eqref{P}. This proves that $\kappa^*>0$, and the proof of Proposition~\ref{Kform} is complete.
\end{proof}
\section{Reduction into weak solutions}\label{weaksolsec}
For later sections, we need to reduce problem~\eqref{P} into an elliptic problem of the weak form. Let $U^\kappa_j$ be as in the proof of Lemma~\ref{L3.1} and set
\begin{gather*}
V^\kappa_j:=U^\kappa_{j+1}-U^\kappa_j\for\kappa>0,\quad j\in\{0,1,\ldots\},\\
w^\kappa_j:=u^\kappa-U^\kappa_j\for \kappa\in\mathcal{K},\quad j\in\{0,1,\ldots\}.
\end{gather*}
We prove the following proposition in this section.
\begin{Proposition}\label{weaksol}
Assume the same conditions as in Theorem~{\rm\ref{Thm1}}. Then there is a number $j_*\in\{1,2,\ldots\}$ with the following properties.
\begin{enumerate}[label={\rm(\roman*)}]
\item 
\begin{equation}\label{w*}
w^\kappa_j\in L^r_\beta\FA j\in\{j_*,j_*+1,\ldots\},\quad r\in(1,\infty),\quad\beta>-1-\frac{1}{r},\quad\kappa\in\mathcal{K}.
\end{equation}
Furthermore, $w^\kappa:=w^\kappa_{j_*+1}$ belongs to $W^{1,2}_0(\R^N_+)$ and it is a weak solution to equation
\begin{equation}\label{wwsol}
-\Delta w^\kappa+w^\kappa=(w^\kappa+U^\kappa_{j_*+1})^p-(U^\kappa_{j_*})^p\jn\R^N_+.
\end{equation}
\item
\begin{equation}\label{V*}
V^\kappa_j\in L^r_\beta\FA j\in\{j_*,j_*+1,\ldots\},\quad r\in(1,\infty),\quad\beta>-1-\frac{1}{r},\quad\kappa>0.
\end{equation}
Furthermore, $V^\kappa:=V^\kappa_{j_*+1}$ belongs to $W^{1,2}_0(\R^N_+)$.
\end{enumerate}
\end{Proposition}
By the same argument as in \cite[Lemma 2.3]{IK}, we obtain
\begin{equation}\label{Umonotone}
U^\kappa_j\le\left(\frac{\kappa}{\kappa'}\right)U^{\kappa'}_j<U^{\kappa'}_j,\quad V^\kappa_j\le\left(\frac{\kappa}{\kappa'}\right)^{(p-1)j+1}V^{\kappa'}_j<V^{\kappa'}_j,
\end{equation}
for any $j\in\{0,1,\ldots\}$ and $0<\kappa<\kappa'$. Furthermore, since $U^\kappa_{j+1}=\Psi(U^\kappa_j,\kappa)$, it follows from \eqref{Phiest0} that
\[
U^\kappa_j\in L^q_\alpha\for\kappa>0,\quad j\in\{0,1,\ldots\}.
\]
Our strategy for the proof of Proposition~\ref{weaksol} is to find a sequence of sets $D_j\subset (1,\infty)\times\R$ such that if $w^\kappa_j\in L^r_\beta$ for any $(r,\beta)\in D_j$, this together with a trivial inequality $w^\kappa_{j+1}\le p(u^\kappa)^{p-1}w^\kappa_j$ implies that $w^\kappa_{j+1}\in L^{r'}_{\beta'}$ for any $(r',\beta')\in D_{j+1}$. Namely,
\begin{Lemma}\label{Dstrategy}
assume the same conditions as in Theorem~{\rm\ref{Thm1}}. Then there is a sequence of sets $D_j\subset (1,\infty)\times\R$, $j\in\{0,1,\ldots\}$, with the following properties.
\begin{enumerate}[label={\rm(\alph*)}]
\item $(q,\alpha)\in D_0.$
\item For any $a\in L^{q/(p-1)}_{(p-1)\alpha}$ and $f\in\bigcap_{(r,\beta)\in D_j}L^r_\beta$, $G[af]$ belongs to $\bigcap_{(r',\beta')\in D_{j+1}}L^{r'}_{\beta'}$.
\item 
\[
D_j=D_*:=\left\{(r,\beta)\in(1,\infty)\times\R: \frac{1}{r}+\beta>-1\right\}
\]
for sufficiently large $j\in\{1,2,\ldots\}$.
\end{enumerate}
\end{Lemma}
We give the proof of Lemma~\ref{Dstrategy} in \ref{Dst}. Immediately from Lemma~\ref{Dstrategy} and induction, the following lemma follows.
\begin{Lemma}\label{apsolests}
Assume the same conditions as in Theorem~{\rm\ref{Thm1}}. Let $D_j\subset(1,\infty)\times\R$ be as in Lemma~{\rm\ref{Dstrategy}}. Then the following assertions hold.
\begin{enumerate}[label={\rm{(\roman*)}}]
\item \begin{equation}\label{iterat}
w^\kappa_j\in L^r_\beta\FA \left(r,\beta\right)\in D_j,\quad\kappa\in\mathcal{K},\quad j\in\{0,1,\ldots\}.
\end{equation}
\item \begin{equation}\label{Viterat}
V^\kappa_j\in L^r_\beta\FA\left(r,\beta\right)\in D_j,\quad\kappa>0,\quad j\in\{0,1,\ldots\}.
\end{equation}
\end{enumerate}
\end{Lemma}
We now complete the proof of Proposition~\ref{weaksol}.
\begin{proof}[Proof of Proposition {\rm\ref{weaksol}}]
Let $j_*\in\{1,2,\ldots\}$ be such that $D_j=D_*$ holds for any $j\in\{j_*,j_*+1,\ldots\}$. Then assertions \eqref{w*} and \eqref{V*} follows from \eqref{Viterat} and \eqref{iterat}.

We next prove that $w^\kappa\in W^{1,2}_0(\R^N_+)$. It suffices to prove that
\begin{equation}\label{U-UW-1r}
(u^\kappa)^p-(U^\kappa_{j_*})^p\in W^{-1,2}(\R^N_+)
\end{equation}
(see Remark~\ref{120}). It follows from \eqref{w*} and the H\"{o}lder inequality that
\begin{equation}\label{sgamma}
(u^\kappa)^p-(U^\kappa_{j_*})^p\le C(u^\kappa)^{p-1}w^\kappa_{j_*}\in L^s_\gamma
\end{equation}
for any $s\in(1,\infty)$ and $\gamma\in\R$ with
\begin{equation}\label{sgammacond}
\frac{1}{s}-\frac{p-1}{q}>0,\quad \gamma>(p-1)\alpha-\left(\frac{1}{s}-\frac{p-1}{q}\right)-1.
\end{equation}
Indeed, letting $1/r:=1/s-(p-1)/q$ and $\beta:=\gamma-(p-1)\alpha$, we observe from \eqref{w*} that $w^\kappa_{j_*}\in L^r_\beta$.

It holds that
\[
\frac{p-1}{q}<\frac{1}{2}+\frac{1}{N}.
\]
Indeed, it is clear if $N=1$. It also follows from \eqref{qNp-12} if $N\ge 2$. We see that
\[
I:=\left(\frac{1}{2},\frac{1}{2}+\frac{1}{N}\right)\cap\left(\frac{p-1}{q},\frac{p-1}{q}+\frac{N}{2(N-1)}\right)\neq\emptyset.
\]
Let $s$ be such that $1/s\in I$, and $\gamma_0:=\max\{(p-1)\alpha-(1/s-(p-1)/q)-1,0\}$. It follows that
\begin{align*}
\frac{1}{2}+\frac{1}{N}\left(2+\frac{1}{s}-(p-1)\left(\frac{1}{q}+\alpha\right)\right)&=\frac{1}{N}\left(2+\frac{N}{2}-\frac{N-1}{s}-(p-1)\left(\frac{1}{q}+\alpha\right)\right)+\frac{1}{s}\\
&>\frac{1}{N}\left(2-(p-1)\left(\frac{N}{q}+\alpha\right)\right)+\frac{1}{s}>\frac{1}{s},
\end{align*}
and thus
\begin{align*}
\frac{1}{2}+\frac{1-\gamma_0}{N}&=\min\left\{\frac{1}{2}+\frac{1}{N}\left(2+\frac{1}{s}-(p-1)\left(\frac{1}{q}+\alpha\right)\right),\frac{1}{2}+\frac{1}{N}\right\}>\frac{1}{s}.
\end{align*}
We find $\gamma$ slightly larger than $\gamma_0$, such that $s$ and $\gamma$ satisfy \eqref{sgammacond} and
\[
\frac{1}{2}<\frac{1}{s}<\frac{1}{2}+\frac{1-\gamma}{N}.
\]
This together with \eqref{sgamma} and Lemma~\ref{Lqa} (ii) implies that \eqref{U-UW-1r}. A similar argument derives that $V^\kappa\in W^{1,2}_0(\R^N_+)$, and the proof of Proposition~\ref{weaksol} is complete.
\end{proof}
\section{Eigenvalue Problems}\label{EVPsec}
We next consider the linearized eigenvalue problem
\begin{equation}\tag{$\mbox{E}_\kappa$}\label{E}
-\Delta\psi+\psi=\lambda p(u^\kappa)^{p-1}\psi\jn\R^N_+,\quad \psi\in W^{1,2}_0(\R^N_+).
\end{equation}
Since
\begin{equation}\label{up-1}
(u^\kappa)^{p-1}\in L^{q/(p-1)}_{(p-1)\alpha},\quad\frac{N(p-1)}{q}<2-(p-1)\alpha,\quad 0\le(p-1)\alpha<2,
\end{equation}
by \eqref{qalpha}, the following assertion holds by the same argument as in \cite{NS01}*{Lemma B.2} (see also Remark~\ref{sigmaRmk}).
\begin{Lemma}\label{EVP}
Assume the same conditions as in Theorem~{\rm\ref{Thm1}}. Let $\kappa\in\mathcal{K}$. Then eigenvalue problem \eqref{E} possesses a first eigenpair $(\psi,\lambda^\kappa)\in W^{1,2}_0(\R^N_+)\times(0,\infty)$. Furthermore,
\[
\integ{\R^N_+}\lambda^\kappa p(u^\kappa)^{p-1}\omega^2dx\le \integ{\R^N_+}(|\nabla\omega|^2+\omega^2)dx\quad\textrm{for any}\quad \omega\in W^{1,2}_0(\R^N_+).
\]
\end{Lemma}
In what follows, we denote by $\lambda^\kappa$ the first eigenvalue of eigenvalue problem \eqref{E} for $\kappa\in\mathcal{K}$. By a standard theory of elliptic eigenvalue problems (see e.g. \cite{GT}*{Section 8.12}. See also Remark~\ref{sigmaRmk}), for any $\kappa\in\mathcal{K}$ there is a unique function $\psi^\kappa\in W^{1,2}_0(\R^N_+)$ such that
\begin{equation}\label{psik}
\begin{gathered}
\left\{\psi\in W^{1,2}_0(\R^N_+): -\Delta\psi+\psi=\lambda^\kappa p(u^\kappa)^{p-1}\psi\jn\R^N_+\right\}=\operatorname{span}
\{\psi^\kappa\},\\
\psi^\kappa>0\jn\R^N_+,\quad \integ{\R^N_+}\left(|\nabla\psi^\kappa|^2+(\psi^\kappa)^2\right)dx=1.
\end{gathered}
\end{equation}
Furthermore, it follows from the same argument as in \cite[Lemma 4.6]{ IOS01} that
\begin{equation}\label{stab}
\lambda^\kappa>1\FA\kappa\in(0,\kappa^*),
\end{equation}
and thus
\begin{equation}\label{Rayleigh}
\integ{\R^N_+}p(u^\kappa)^{p-1}\omega^2dx\le\integ{\R^N_+}(|\nabla\omega|^2+\omega^2)dx\FA\omega\in W^{1,2}_0(\R^N_+),\quad \kappa\in(0,\kappa^*).
\end{equation}
\begin{Lemma}\label{psireg}
Assume the same conditions as in Theorem~{\rm\ref{Thm1}}. Let $\kappa\in\mathcal{K}$. Then $\psi^\kappa$ satisfies
\begin{equation}\label{psiIE}
\psi^\kappa=G[p(u^\kappa)^{p-1}\psi^\kappa].
\end{equation}
Furthermore,
\begin{equation}\label{psiinteg}
\psi^\kappa\in L^r_\beta\FA r>1,\quad\beta>-\frac{1}{r}-1.
\end{equation}
\end{Lemma}
\begin{proof}
It follows from $(u^\kappa)^{p-1}\in L^{q/(p-1)}_{(p-1)\alpha}$ and \eqref{up-1} that $(u^\kappa)^{p-1}\psi\in W^{-1,2}(\R^N_+)$ (see Remark~\ref{120to-12}). Assertion \eqref{psiIE} follows from this and Proposition~\ref{G120}.

Furthermore, it follows from elliptic regularity theorems (see e.g. \cite{GT}*{Theorem 8.17}. See also Remark~\ref{sigmaRmk}) that $\psi^{\kappa}\in L^2(\R^N_+)\cap L^\infty(\R^N_+)\subset L^{r_0}(\R^N_+)$. Here $r_0\in[2,\infty)$ satisfies
\[
\frac{1}{r_0}<\min\left\{1-\frac{p-1}{q},2-(p-1)\left(\frac{1}{q}+\alpha\right)\right\}.
\]
This together with \eqref{psiIE} and Proposition~\ref{Dstrategy'} implies that
\[
\psi^\kappa\in L^r_\beta\FA (r,\beta)\in D_j\left(r_0,0\right),\quad j\in\{1,2,\ldots\},
\]
where $D_j(r_0,0)$ is as in \ref{Dst}. Taking sufficiently large $j$, we obtain \eqref{psiinteg} and complete the proof of Lemma~\ref{psireg}.
\end{proof}

Now we are ready to prove Theorem~\ref{Thm1}.
\begin{proof}[Proof of Theorem~{\rm\ref{Thm1}}]
By Proposition~\ref{Kform}, it remains to prove that $\kappa^*<\infty$. Let $\kappa\in(0,\kappa^*)$. It follows from the definition of solutions to problem~\eqref{P} that $\kappa P[\mu]\le u^\kappa$ a.e. in $\R^N_+$. This together with \eqref{Rayleigh} implies that
\[
\kappa^{p-1}\le\inf_{\substack{\omega\in W^{1,2}_0(\R^N_+),\\ \omega\not\equiv 0}}\frac{\integ{\R^N_+}(|\nabla\omega|^2+\omega^2)dx}{\integ{\R^N_+}pP[\mu]^{p-1}\omega^2dx}<\infty.
\]
This implies that $\kappa^*<\infty$ and completes the proof of Theorem~\ref{Thm1}.
\end{proof}
\section{Uniform estimates of $w^\kappa$}\label{unformest}
In this section, we give uniform estimates of $\{w^\kappa\}_{\kappa\in(0,\kappa^*)}$ to prove the following proposition.
\begin{Proposition}\label{l*existence}
Assume the same conditions as in Theorem~{\rm\ref{Thm2}}. Then $\kappa^*\in\mathcal{K}$.
\end{Proposition}
Our uniform estimates are based on a uniform energy estimate for $(w^\kappa)^\nu$ and elliptic regularity theorems.
\begin{Lemma}
Assume the same conditions as in Theorem~{\rm\ref{Thm1}}. Let $\nu\ge 1$ be such that
\begin{equation}\label{nuassump1}
\frac{\nu^2}{2\nu-1}<p.
\end{equation}
Then
\begin{equation}\label{NRGest}
\sup_{\kappa\in(0,\kappa^*)}\integ{\R^N_+}\left(|\nabla(w^\kappa)^\nu|^2+(w^\kappa)^{2\nu}\right)dx<\infty.
\end{equation}
\end{Lemma}
\begin{proof}
We first observe that
\begin{equation}\label{nu22nu-1}
\frac{\nu^2}{2\nu-1}=1+\frac{(\nu-1)^2}{2\nu-1}\ge 1\FA \nu\ge 1.
\end{equation}
Let $\kappa\in(0,\kappa^*)$. It follows from Lemma~\ref{w*} that $w^\kappa,V^\kappa\in W^{1,2}_0(\R^N_+)\cap L^\infty(\R^N_+)$ and thus
\[
(w^\kappa)^k\in W^{1,2}_0(\R^N_+),\quad\nabla(w^\kappa)^k=k(w^\kappa)^{k-1}\nabla w^\kappa,
\]
\[
(V^\kappa)^k\in W^{1,2}_0(\R^N_+),\quad\nabla(V^\kappa)^k=k(V^\kappa)^{k-1}\nabla V^\kappa,
\]
for any $k\ge 1$. Let $\varepsilon>0$ and $\delta>0$. It follows from \eqref{nu22nu-1} that
\begin{equation}\label{NRG1}
\begin{aligned}
\integ{\R^N_+}\left(|\nabla(w^\kappa)^\nu|^2+(w^\kappa)^{2\nu}\right)dx&=\integ{\R^N_+}\left(\nu^2(w^\kappa)^{2\nu-2}\nabla |w^\kappa|^2+(w^\kappa)^{2\nu}\right)dx\\
&\le\frac{\nu^2}{2\nu-1}\integ{\R^N_+}\left(\nabla w^\kappa\cdot\nabla(w^\kappa)^{2\nu-1}+w^\kappa(w^\kappa)^{2\nu-1}\right)dx.
\end{aligned}
\end{equation}
Furthermore, it follows from \eqref{wwsol}, \eqref{Umonotone}, and \cite[Lemma 5.1]{IOS01} that
\begin{equation}\label{NRG2}
\begin{aligned}
&\integ{\R^N_+}\left(\nabla w^\kappa\cdot\nabla(w^\kappa)^{2\nu-1}+(w^\kappa)^{2\nu}\right)dx=\integ{\R^N_+}((u^\kappa)^p-(U^\kappa_{j_*})^p)(w^\kappa)^{2\nu-1}dx\\
&\le\integ{\R^N_+}\left((1+\varepsilon)(u^\kappa)^{p-1}(w^\kappa+V^\kappa)+C(U^\kappa_{j_*})^{p-1+2\nu\delta}(w^\kappa+V^\kappa)^{1-2\nu\delta}\right)(w^\kappa)^{2\nu-1}dx\\
&\begin{multlined}\le\integ{\R^N_+}\left((u^\kappa)^{p-1}\left((1+\varepsilon)^2(w^\kappa)^{2\nu}+C(V^{\kappa^*})^{2\nu}\right)+C(U^\kappa_{j_*})^{p-1+2\nu\delta}\left(w^\kappa+V^{\kappa^*}\right)^{2\nu(1-\delta)}\right)dx.\end{multlined}
\end{aligned}
\end{equation}
Here, \eqref{Rayleigh} implies that
\begin{equation}\label{NRG3}
\begin{gathered}
\integ{\R^N_+}(u^\kappa)^{p-1}(w^\kappa)^{2\nu}dx\le\frac{1}{p}\integ{\R^N_+}\left(|\nabla(w^\kappa)^\nu|^2+(w^\kappa)^{2\nu}\right)dx,\\
\integ{\R^N_+}(u^\kappa)^{p-1}(V^{\kappa^*})^{2\nu}dx\le\frac{1}{p}\integ{\R^N_+}\left(|\nabla(V^{\kappa^*})^\nu|^2+(V^{\kappa^*})^{2\nu}\right)dx\le C.
\end{gathered}
\end{equation}
On the other hand, we set
\[
\frac{1}{r_{\nu,\delta}}:=\frac{1}{2(1-\delta)}\left(1-\frac{p-1+2\nu\delta}{q}\right),\quad \beta_{\nu,\delta}:=\frac{p-1+2\nu\delta}{2(1-\delta)}\alpha,
\]
Since it follows from \eqref{qalpha} that
\[
\frac{1}{2}\left(1-\frac{p-1}{q}\right)\in\left(\frac{1}{2}-\frac{1}{N}\left(1-\frac{(p-1)\alpha}{2}\right),\frac{1}{2}\right),
\]
\[
\frac{1}{2\nu}\left(1-\frac{p-1}{q}\right)-\frac{p-1}{2\nu}\alpha=\frac{1}{\nu}\left(\frac{1}{2}-\frac{p-1}{2}\left(\frac{1}{q}+\alpha\right)\right)>\frac{1}{\nu}\left(\frac{1}{2}-\frac{p-1}{p}\right)>-\frac{1}{2\nu}>-1,
\]
we find a sufficiently small $\delta>0$ such that
\[
r_{\nu,\delta}\in\left(\frac{1}{2}-\frac{1-\beta_{\nu,\delta}}{N},\frac{1}{2}\right),\quad \frac{1}{\nu r_{\nu,\delta}}+\frac{\beta_{\nu,\delta}}{\nu}>-1.
\]
This together with \eqref{W1toLqa}, \eqref{V*}, and the H\"{o}lder inequality implies that
\begin{equation}\label{NRG4}
\begin{multlined}
\integ{\R^N_+}(U^\kappa_{j_*})^{p-1+2\nu\delta}(w^\kappa+V^{\kappa^*})^{2\nu(1-\delta)}\,dx\le\|U^\kappa_{j_*}\|_{L^q_\alpha}^{p-1+2\nu\delta}\left(\|(w^\kappa)^\nu\|_{L^{r_{\nu,\delta}}_{-\beta_{\nu,\delta}}}^{2(1-\delta)}+\|V^{\kappa^*}\|_{L^{\nu r_{\nu,\delta}}_{-\beta_{\nu,\delta}/\nu}}^{2\nu(1-\delta)}\right)\\
\le C\|U^{\kappa^*}_{j_*}\|_{L^q_\alpha}^{p-1+2\nu\delta}\left(\|(w^\kappa)^\nu\|_{W^{1,2}(\R^N_+)}^{2(1-\delta)}+C\right)\le \varepsilon'\integ{\R^N_+}(|\nabla (w^\kappa)^\nu|^2+(w^\kappa)^{2\nu})dx+C
\end{multlined}
\end{equation}
for any $\varepsilon'>0$.

Combining \eqref{NRG1}, \eqref{NRG2}, \eqref{NRG3}, and \eqref{NRG4}, we deduce that
\[
\integ{\R^N_+}\left(|\nabla(w^\kappa)^\nu|^2+(w^\kappa)^{2\nu}\right)dx\le\left(\frac{(1+\varepsilon)^2\nu^2}{p(2\nu-1)}+C\varepsilon'\right)\integ{\R^N_+}\left(|\nabla(w^\kappa)^\nu|^2+(w^\kappa)^{2\nu}\right)dx+C.
\]
This together with \eqref{nuassump1} and an appropriate choice of $\varepsilon$ and $\varepsilon'$ implies \eqref{NRGest}.
\end{proof}
\begin{Lemma}
Assume the same conditions as in Theorem~{\rm\ref{Thm1}} and $1<p<p_{JL}$. Then
\begin{equation}\label{unif}
\sup_{\kappa\in(0,\kappa^*)}\|w^\kappa\|_{L^\infty(\R^N)}<\infty.
\end{equation}
\end{Lemma}
\begin{proof}
By \cite{F}*{Proof of Theorem 1} with a change of variable $\gamma=2\nu-1$,
\[
L:=\left(\frac{(N-2)(p-1)}{4},\infty\right)\cap[1,p+\sqrt{p^2-p})\neq\emptyset
\]
if and only if $1<p<p_{JL}$. Since
\[
\frac{d}{d\nu}\frac{\nu^2}{2\nu-1}=\left(\frac{2}{\nu}-\frac{2}{2\nu-1}\right)\frac{\nu^2}{2\nu-1}>0\for\nu>1,
\]
$\nu^2/(2\nu-1)$ is strictly increasing for $\nu\ge 1$. This together with
\[
\left.\frac{\nu^2}{2\nu-1}\right|_{\nu=p+\sqrt{p^2-p}}=\frac{2p^2-p+2p\sqrt{p^2-p}}{2p-1+2\sqrt{p^2-p}}=p
\]
implies that
\[
\frac{\nu^2}{2\nu-1}<p\FA \nu\in L.
\]
Let $\nu\in L$. It follows from \eqref{NRGest} and the Sobolev inequality that
\[
\sup_{\kappa\in(0,\kappa^*)}\|w^\kappa\|_{L^{2N\nu/(N-2)}(\R^N_+)}<\infty.
\]
This together with $\nu\in L$, \eqref{V*}, and \eqref{Umonotone} implies that
\[
\sup_{\kappa\in(0,\kappa^*)}\|(w^\kappa)^{p-1}\|_{L^{2N\nu/((N-2)(p-1))}(\R^N_+)}<\infty,\quad \frac{2N}{(N-2)(p-1)}>\frac{N}{2},
\]
\[
\sup_{\kappa\in(0,\kappa^*)}\|(U^\kappa_{j_*})^{p-1}\|_{L^{q/(p-1)}_{(p-1)\alpha}}\le\|U^{\kappa^*}_{j_*}\|_{L^q_\alpha}^{p-1}<\infty,\quad \sup_{\kappa\in(0,\kappa^*)}\|V^\kappa\|_{L^\infty}\le\|V^{\kappa^*}\|_{L^\infty}<\infty.
\]
Since
\[
0<-\Delta w^\kappa\le p(w^\kappa+U^\kappa_{j_*})^{p-1}(w^\kappa+V^\kappa)\jn\R^N_+,
\]
assertion \eqref{unif} follows from elliptic regularity theorems (see e.g. \cite{GT}*{Theorem 8.17}. See also Remark~\ref{sigmaRmk}).
\end{proof}
Now we complete the proof of Proposition~\ref{l*existence}.
\begin{proof}[Proof of Proposition {\rm\ref{l*existence}}]
It follows from \eqref{NRGest} and \eqref{unif} that the limit
\[
w^*(x):=\lim_{\kappa\to\kappa^*-0}w^\kappa(x)
\]
exists for a.a. $x\in\R^N_+$, $w^*\in L^2(\R^N_+)\cap L^\infty(\R^N_+)$, and $w^*$ is a solution to integral equation \eqref{IE} with $\kappa=\kappa^*$. By the same argument as in Lemma~\ref{psireg}, it follows that $w^{\kappa^*}\in L^q_\alpha$ and the proof of Proposition~\ref{l*existence} is complete.
\end{proof}
\section{Proof of Theorem~\ref{Thm2}}\label{PfThm2}
In this section, we complete the proof of Theorem~\ref{Thm2}. In particular, we prove the uniqueness of an $L^q_\alpha$-solution in the case $\kappa=\kappa^*$, and the multiplicity assertion (iii).

We first characterize $\kappa^*$ using the first eigenvalue $\lambda^\kappa$ of eigenvalue problem \eqref{E}. Let $\Psi(v,\kappa)$ be as in \eqref{Psi} and set
\[
\Phi(v,\kappa):=v-\Psi(v,\kappa)\for (v,\kappa)\in L^q_\alpha\times(0,\infty).
\]
It follows from \eqref{Phiest0} that $\Phi: L^q_\alpha\times(0,\infty)\to L^q_\alpha$. It is also clear that $\Phi\in C^1(L^q_\alpha\times(0,\infty),L^q_\alpha)$ and
\begin{equation}\label{Phiv}
\Phi_v(v,\kappa)h=h-G[pv_+^{p-1}h]\for h\in L^q_\alpha,\quad (v,\kappa)\in L^q_\alpha\times(0,\infty).
\end{equation}

\begin{Lemma}\label{l*=1}
Assume the same conditions as in Theorem~{\rm\ref{Thm2}} and let $\kappa\in\mathcal{K}$. If $\lambda^\kappa>1$, then $\kappa<\kappa^*$. Furthermore, $\lambda^{\kappa^*}=1$.
\end{Lemma}
\begin{proof}
It suffices to prove that $\Phi_v(u^\kappa,\kappa):L^q_\alpha\to L^q_\alpha$ is an isomorphism of Banach spaces if $\lambda^{\kappa}>1$. Indeed, this together with the implicit function theorem implies that $\Phi(U,\kappa')=0$ for some $U\in L^q_\alpha$ and $\kappa'>\kappa$, which implies that $\kappa<\kappa'\le\kappa^*$. Furthermore, since $\kappa^*\ge 1$ by the same argument as in \cite[Lemma 6.1]{IOS01}, $\lambda^{\kappa^*}=1$ holds.

We first prove that $\Ker(\Phi_v(u^\kappa,\kappa))=\{0\}$. Let $h\in\Ker(\Phi_v(u^\kappa,\kappa))$. It follows from \eqref{Phiv} that $h=G[p(u^\kappa)^{p-1}h]$. Applying the same argument as in Section~\ref{weaksolsec}, we obtain
\[
h\in W^{1,2}_0(\R^N_+),\quad -\Delta h+h=p(u^\kappa)^{p-1}h\jn\R^N_+.
\]
This together with $\lambda^\kappa>1$ implies that $h\equiv 0$.

It remains to prove that $\Im(\Phi_v(u^\kappa,\kappa))=L^q_\alpha$. It follows from \eqref{qalpha} that
\begin{gather*}
\frac{1}{q}+\alpha>0>-1,\quad\frac{1}{q}+\frac{p-1}{q}=\frac{p}{q}<1,\\
\frac{N(p-1)}{q}+(p-1)\alpha=(p-1)\left(\frac{N}{q}+\alpha\right)<2,\\
\frac{1}{q}+\frac{p-1}{q}+\alpha+(p-1)\alpha=p\left(\frac{1}{q}+\alpha\right)<2.
\end{gather*}
This together with $p(u^\kappa)^{p-1}\in L^{q/(p-1)}_{(p-1)\alpha}$ and Proposition~\ref{cpt} implies that $h\mapsto G[p(u^\kappa)^{p-1}h]$ is a compact operator from $L^q_\alpha$ to itself. It follows from the Fredholm alternative theorem and  $\Ker(\Phi_v(u^\kappa,\kappa))=\{0\}$ that $\Im(\Phi_v(u^\kappa,\kappa))=L^q_\alpha$, and the proof of Lemma~\ref{l*=1} is complete.
\end{proof}
\begin{Lemma}\label{1dim}
Assume the same conditions as in Theorem~{\rm\ref{Thm2}}. Then
\begin{gather}
\Ker(\Phi_v(u^{\kappa^*},\kappa^*))=\operatorname{span}\{\psi^{\kappa^*}\},\label{l*ker}\\
\Im(\Phi_v(u^{\kappa^*},\kappa^*))=\left\{f\in L^q_\alpha: \integ{\R^N_+}p(u^{\kappa^*})^{p-1}\psi^{\kappa^*}fdx=0\right\}.\label{l*im}
\end{gather}
Here $\psi^{\kappa^*}$ is as in \eqref{psik}.
\end{Lemma}
\begin{proof}
We first prove \eqref{l*ker}. Let $h\in\Ker(\Phi_v(u^{\kappa^*},\kappa^*))$. Applying the same argument as in the proof of Lemma~\ref{l*=1}, we observe that $h\in W^{1,2}_0(\R^N_+)$ and $h$ is a weak solution to equation
\[
-\Delta h+h=p(u^{\kappa^*})^{p-1}h\jn\R^N_+.
\]
Thus $h$ is an eigenfunction of eigenvalue problem \eqref{E} with $\kappa=\kappa^*$ corresponding to the eigenvalue $\lambda^{\kappa^*}=1$, and hence $h\in\operatorname{span}\{\psi^{\kappa^*}\}$ by \eqref{psik}. These imply that 
\[
\Ker(\Phi_v(u^{\kappa^*},\kappa^*))\subset\operatorname{span}\{\psi^{\kappa^*}\}.
\]
Conversely, since $\psi^{\kappa^*}\in L^q_\alpha$ and $\psi^{\kappa^*}=G[p(u^\kappa)^{p-1}\psi^{\kappa^*}]$ by \eqref{psiIE} and \eqref{psiinteg}, $\psi^{\kappa^*}$ belongs to $\Ker(\Phi_v(u^{\kappa^*},\kappa^*))$. Thus \eqref{l*ker} follows.

We next prove \eqref{l*im}. Since
\[
\frac{q-p}{q}-p\alpha=1-p\left(\frac{1}{q}+\alpha\right)>1-p\cdot\frac{2}{p}=-1
\]
by \eqref{qalpha}, it follows from \eqref{psiinteg} that  $\psi^{\kappa^*}\in L^{q/(q-p)}_{-p\alpha}$. This together with $(u^{\kappa^*})^{p-1}\in L^{q/(p-1)}_{(p-1)\alpha}$ implies that $p(u^{\kappa^*})^{p-1}\psi^{\kappa^*}\in L^{q/(q-1)}_{-\alpha}=(L^q_\alpha)^*,$ and thus
\begin{equation}\label{Z}
Z:=\left\{f\in L^q_\alpha: \integ{\R^N_+}p(u^{\kappa^*})^{p-1}\psi^{\kappa^*}fdx=0\right\}
\end{equation}
is a well-defined closed subspace of $L^q_\alpha$ and
\begin{equation}\label{RHScodim}
\operatorname{codim} Z=1.
\end{equation}

On the other hand, by the same argument as in the proof of Lemma~\ref{l*=1}, $h\mapsto G[p(u^{\kappa^*})^{p-1}h]$ is a compact operator from $L^q_\alpha$ to itself. It follows from the Fredholm alternative theorem and \eqref{l*ker} that
\begin{equation}\label{codim}
\operatorname{codim}\left(\Im(\Phi_v(u^{\kappa^*},\kappa^*))\right)=\dim\left(\Ker(\Phi_v(u^{\kappa^*},\kappa^*))\right)=1.
\end{equation}
We next prove that
\begin{equation}\label{l*im'}
\Phi_v(u^{\kappa^*},\kappa^*)(C^\infty_c(\R^N_+))\subset Z.
\end{equation}
Let $f\in\Phi_v(u^{\kappa^*},\kappa^*)(C^\infty_c(\R^N_+))$ and $h\in C^\infty_c(\R^N_+)$ be such that $\Phi_v(u^{\kappa^*},\kappa^*)h=f$, that is,
\[
h-f=G[p(u^{\kappa^*})^{p-1}h].
\]
Since $h\in L^r_\beta$ for any $(r,\beta)\in D_*$, by the same argument as in the proof of Proposition~\ref{weaksol}, we deduce that $v:=h-f\in W^{1,2}_0(\R^N_+)$ and $v$ is a weak solution to equation $-\Delta v+v=p(u^{\kappa^*})^{p-1}h$~in~$\R^N_+$. Using this and \eqref{psik}, we obtain
\[
\integ{\R^N_+}p(u^{\kappa^*})^{p-1}h\psi^{\kappa^*}dx=\integ{\R^N_+}(\nabla\psi^{\kappa^*}\cdot\nabla v+\psi^{\kappa^*}v)dx=\integ{\R^N_+}p(u^{\kappa^*})^{p-1}\psi^{\kappa^*}vdx,
\]
and thus
\[
\integ{\R^N_+}p(u^{\kappa^*})^{p-1}\psi^{\kappa^*}fdx=0.
\]
This proves \eqref{l*im'}.

Since $C^\infty_c(\R^N_+)$ is a dense subspace of $L^q_\alpha$, $\Phi_v(u^{\kappa^*},\kappa^*)(C^\infty_c(\R^N_+))$ is a dense subspace of $\Im(\Phi_v(u^{\kappa^*},\kappa^*))$. By taking closures of both sides of \eqref{l*im'} in $L^q_\alpha$, we obtain
\[
\Im(\Phi_v(u^{\kappa^*},\kappa^*))\subset Z.
\]
This together with \eqref{RHScodim} and \eqref{codim} implies \eqref{l*im}, the proof of Lemma~\ref{1dim} is complete.
\end{proof}
Now we are ready to prove Theorem~\ref{Thm2}.
\begin{proof}[Proof of Theorem~{\rm\ref{Thm2}}]
We first complete the proof of assertions (i) and (ii). By Theorem~\ref{Thm1} and Proposition~\ref{l*existence}, it suffices to prove that $u^{\kappa^*}$ is the unique $L^q_\alpha$-solution to problem~\eqref{P} with $\kappa=\kappa^*$. Let $\tilde{u}$ be an $L^q_\alpha$-solution to problem~\eqref{P} with $\kappa=\kappa^*$. It follows from the minimality of $u^{\kappa^*}$ that $u^{\kappa^*}\le\tilde{u}$ a.e. in $\R^N_+$. Furthermore, by the same argument as in the proof of Proposition~\ref{weaksol}, it follows that $\tilde{w}:=\tilde{u}-U^{\kappa^*}_{j_*+1}\in W^{1,2}_0(\R^N_+)$ and $\tilde{w}$ is a weak solution to equation $-\Delta\tilde{w}+\tilde{w}=(\tilde{u})^p-(U^{\kappa^*}_{j_*})^p$~in~$\R^N_+$. 
Hence $z:=\tilde{u}-u^{\kappa^*}\in W^{1,2}_0(\R^N_+)$ and $z$ is a weak solution to equation
\[
-\Delta z+z=(\tilde{u})^p-(u^{\kappa^*})^p\jn\R^N_+.
\]
This together with Lemma~\ref{l*=1} implies that
\[
\integ{\R^N_+}p(u^{\kappa^*})^{p-1}z\psi^{\kappa^*}dx=\integ{\R^N_+}(\nabla z\cdot\nabla\psi^{\kappa^*}+z\psi^{\kappa^*})dx=\integ{\R^N_+}((\tilde{u})^p-(u^{\kappa^*})^p)\psi^{\kappa^*} dx.
\]
Since $s^p-t^p>ps^{p-1}(s-t)$ for any $s>t>0$, it follows that $\tilde{u}=u^{\kappa^*}$ a.e. in $\R^N_+$, and the proof of assertions (i) and (ii) is complete.

We finally prove assertion (iii). Let $Z$ be as in \eqref{Z}. Then $Z$ is a complement subspace of $\operatorname{span}\{\psi^{\kappa^*}\}$~in~$L^q_\alpha$. Since $\Phi_\kappa(u^{\kappa^*},\kappa^*)=P[\mu]>0$~in~$\R^N_+$,
\[
\integ{\R^N_+}p(u^{\kappa^*})^{p-1}\psi^{\kappa^*}\Phi_\kappa(u^{\kappa^*},\kappa^*)dx>0.
\]
It follows from \eqref{l*im} that $\Phi_\kappa(u^{\kappa^*},\kappa^*)\notin\Im\left(\Phi_v(u^{\kappa^*},\kappa^*)\right)$. This together with \eqref{l*ker} and \eqref{l*im} enables us to apply the following bifurcation theorem.
\begin{Proposition}\cite[Theorem 3.2]{CR1}
Let $X$, $Y$ be two Banach spaces, $U\subset X\times\R$ be an open set, and $F=F(u,\kappa):U\to Y$ is a $C^1$ map. Let $(u^*,\kappa^*)\in U$ be such that
\begin{gather*}
F(u^*,\kappa^*)=0,\quad \dim\left(\Ker (F_u(u^*,\kappa^*))\right)=\operatorname{codim} (\Im F_u(u^*,\kappa^*))=1,\\
F_\kappa(u^*,\kappa^*)\notin \Im (F_u(u^*,\kappa^*)).
\end{gather*}
Let $\psi\in X$ be such that $\Ker(F_u(u^*,\kappa^*))=\operatorname{span}\{\psi\}$, and $Z\subset X$ be a complemented subspace of $\Ker(F_u(u^*,\kappa^*))$. Then solutions $(u,\kappa)$ to equation $F(u,\kappa)=0$ near $(u^*,\kappa^*)$ form a $C^1$ curve
\[
(u(s),\kappa(s))=(u^*+s\psi+z(s),\kappa^*+\tau(s)),
\]
with $(z,\tau):(-\delta,\delta)\to Z\times\R$, $z(0)=0$,  $\tau(0)=0$, and $\tau'(0)=0$.
\end{Proposition}
Using this theorem with $F=\Phi$ and $X=Y=L^q_\alpha$, we obtain a $C^1$ curve $(u(s),\kappa(s))=(u^{\kappa^*}+s\psi^{\kappa^*}+z(s),\kappa^*+\tau(s))$ with $(z,\tau):(-\delta,\delta)\to Z\times\R$ with small enough $\delta>0$, such that
\[
z(0)=0,\quad \tau(0)=0,\quad \tau'(0)=0,
\]
\begin{equation}\label{curve}
\Phi(u(s),\kappa(s))=0\quad\textrm{for any}\quad s\in(-\delta,\delta).
\end{equation}
Here we see that $u(s_1)\neq u(s_2)$ for $s_1,s_2\in(-\delta,\delta)$ with $s_1\neq s_2$. Indeed, since $u(s_1)=u^{\kappa^*}+s_1\psi^{\kappa^*}+z(s_1)$, $u(s_2)=u^{\kappa^*}+s_2\psi^{\kappa^*}+z(s_2)$, $z(s_1)$ and $z(s_2)$ belong a complemented subspace $Z$ of $\operatorname{span}\{\psi^{\kappa^*}\}$, it holds that $u(s_1)\neq u(s_2)$.

It follows from assertion (ii) in Theorem~\ref{Thm1} that $\kappa(s)\le\kappa^*$ for any $s\in(-\delta,\delta)$. Furthermore, since $u^{\kappa^*}$ is the unique $L^q_\alpha$-solution to problem~\eqref{P} with $\kappa=\kappa^*$ and $\kappa(0)=\kappa^*$, $\kappa(s)<\kappa^*$ for any $s\in(-\delta,\delta)\setminus\{0\}$. By the intermediate value theorem, for sufficiently small $\varepsilon>0$, there are two numbers $s_1\in(-\delta,0)$ and $s_2\in(0,\delta)$ such that $\kappa(s_1)=\kappa(s_2)=\kappa^*-\varepsilon$. This together with \eqref{curve} and $s_1<0<s_2$ implies that $u(s_1)$ and $u(s_2)$ are two distinct $L^q_\alpha$-solutions to problem~\eqref{P} with $\kappa=\kappa^*-\varepsilon$. This proves assertion (iii), and the proof of Theorem~\ref{Thm2} is complete.
\end{proof}
\appendix
\def\thesection{Appendix \Alph{section}}
\section{On the assumption \eqref{pmu}}\label{pmucond}
\def\thesection{\Alph{section}}
In this section, we give some notes on the assumption \eqref{pmu}. We first summarize properties of the integral kernel $P(x,z)$.
\begin{Lemma}\label{Pprop}
\begin{enumerate}[label={\rm(\roman*)}]
\item Let $x,y\in\R^N_+$ and $z\in\R^{N-1}$. Then
\begin{equation}\label{Psemig}
P(x+y,z)=\int_{\R^{N-1}}P(x,z-\zeta)P(y,\zeta)d\zeta.
\end{equation}
\item Let $x\in\R^N_+$. Then
\begin{equation}\label{Pint}
\int_{\R^{N-1}}P(x,z)dz=e^{-x_N}.
\end{equation}
\item Let
\[
\tilde{P}(x,z):=\Gamma\left(\frac{N}{2}\right)\pi^{-N/2}x_N|(x'-z,x_N)|^{-N}\for x\in\R^N_+,\quad z\in\R^{N-1}
\]
be the Poisson kernel for $-\Delta$. Then
\begin{equation}\label{Pcomp1}
P(x,z)\le C\tilde{P}(x,z)\FA x\in\R^N_+,\quad z\in\R^{N-1}.
\end{equation}
\item Let $\eta\in C^\infty_c(\R^{N-1})$ be such that $0\le\eta(z)\le 1$ for $z\in\R^{N-1}$, $\eta(z)=0$ for $|z|\ge 1$, and $\eta(z)=1$ for $|z|\le 1/2$. Set
\[
\overline{P}(x,z):=x_N^{1-N}\eta\left(\frac{x'-z}{x_N}\right).
\]
Then
\begin{equation}\label{Pcomp2}
\overline{P}(x,z)\le CP(x,z)\FA x\in\R^N_+\with x_N<1,\quad z\in\R^{N-1}.
\end{equation}
\end{enumerate}
\end{Lemma}
\begin{proof}
Equalities \eqref{Psemig} and \eqref{Pint} are consequences of equality \eqref{GRepr}.

We next prove \eqref{Pcomp1}. It follows from the fact
\[
|E'(r)|\asymp r^{1-N}\as r\to+0,\quad |E'(r)|\asymp r^{-(N-1)/2}e^{-r}\as r\to+\infty,
\]
that
\begin{equation}\label{Pequiv}
P(x,z)\approx\begin{cases}
x_N(|x'-z|^2+x_N^2)^{-N/2}&\jf |x'-z|^2+x_N^2\le 2,\\
x_N(|x'-z|^2+x_N^2)^{-(N+1)/4}\exp(-(|x'-z|^2+x_N^2)^{1/2})&\jf |x'-z|^2+x_N^2>2.
\end{cases}
\end{equation}
This together with the definition of $\tilde{P}(x,z)$ implies \eqref{Pcomp1}.

We finally prove \eqref{Pcomp2}. Let $x\in\R^N_+$ be such that $x_N<1$, and $z\in\R^{N-1}$. If $|x'-z|<x_N$, then $|x'-z|^2+x_N^2<2$. This together with \eqref{Pequiv} implies that
\[
\overline{P}(x,z)\le x_N^{1-N}\le Cx_N(|x'-z|^2+x_N^2)^{-N/2}\le CP(x,z).
\]
On the other hand, if $|x'-z|\ge x_N$, then $\overline{P}(x,z)=0$. Thus \eqref{Pcomp2} follows.
\end{proof}
Now we give the following Besov space characterization of $\mu$ with $P[\mu]\in L^q_\alpha$.
\begin{Proposition}\label{Bsqq}
Let $q\in(1,\infty)$, $\alpha\ge 0$, and $\mu$ be a nonnegative Radon measure on $\R^{N-1}$. Then the following conditions are equivalent.
\begin{enumerate}[label={\rm(\roman*)}]
\item $P[\mu]\in L^q_\alpha$.
\item \begin{equation}\label{stripe}
\int_{\R^{N-1}\times(0,1)}P[\mu](x)^qx_N^{q\alpha} dx<\infty.
\end{equation}
\item $\mu\in B^{-1/q-\alpha}_{q,q}(\R^{N-1})$.
\end{enumerate}
\end{Proposition}
\begin{proof}
(i) $\implies$ (ii) is clear. Conversely, if $\mu$ satisfies (ii), we find $\tau\in(0,1)$ such that
\[
\int_{\R^{N-1}}P[\mu](z,\tau)^qdz<\infty.
\]
This together with \eqref{Psemig} and \eqref{Pint} implies that
\begin{multline*}
\left(\int_{\R^{N-1}}P[\mu](z,\tau+t)^qdz\right)^{1/q}\le \left(\int_{\R^{N-1}}\left(\int_{\R^{N-1}}P(z-\zeta,t)P[\mu](\zeta,\tau)d\zeta\right)^qdz\right)^{1/q}\\
\le \left(\int_{\R^{N-1}}P(\zeta,t)d\zeta\right)\left(\int_{\R^{N-1}}P[\mu](z,\tau)^qdz\right)^{1/q}=e^{-t}\left(\int_{\R^{N-1}}P[\mu](z,\tau)^qdz\right)^{1/q}
\end{multline*}
for any $t>0$. This together with \eqref{stripe} implies that $P[\mu]\in L^q_\alpha$, and we obtain (ii) $\implies$ (i).

We next prove that (ii) $\iff$ (iii). Let $\tilde{P}$ and $\overline{P}$ be as in Lemma \ref{Pprop}. Set
\[
\tilde{P}[\mu](x):=\int_{\R^{N-1}}\tilde{P}(x,z)d\mu(z),\quad\overline{P}[\mu](x):=\int_{\R^{N-1}}\overline{P}[\mu](x,z)d\mu(z),
\]
for $x\in\R^N_+$. It follows from \eqref{Pcomp1} and \eqref{Pcomp2} that
\[
C\overline{P}[\mu]\le P[\mu]\le C\tilde{P}[\mu]\jn\R^N_+.
\]
Furthermore, it follows from \cite{T}*{Theorem 3} that
\[
\int_{\R^{N-1}\times(0,1)}\tilde{P}[\mu]^qx_N^{q\alpha} dx<\infty\iff\int_{\R^{N-1}\times(0,1)}\overline{P}[\mu]^qx_N^{q\alpha} dx<\infty\iff \mu\in B^{-1/q-\alpha}_{q,q}.
\]
Thus (ii) $\iff$ (iii) holds, and the proof of Proposition~\ref{Bsqq} is complete.
\end{proof}
As an immediate consequence of Proposition~\ref{Bsqq}, we obtain the following condition equivalent to the assumption in Theorem~\ref{Thm1}.
\begin{Proposition}\label{Equivassump}
Let $p>1$ and $\mu$ be a nonnegative Radon measure on $\R^{N-1}$. The following conditions are equivalent.
\begin{enumerate}[label={\rm(\roman*)}]
\item $P[\mu]\in L^q_\alpha$ for some $q\in(p,\infty)$ and $\alpha\ge 0$ satisfying \eqref{qalpha}.
\item \begin{equation}\label{equiv}
\mu\in B^{-s}_{q,q}(\R^{N-1})\FS q>\max\left\{p,\frac{N}{2}(p-1)\right\},\quad s<\min\left\{\frac{2}{p},\frac{2}{p-1}-\frac{N-1}{q}\right\}.
\end{equation}
\end{enumerate}
\end{Proposition}
\begin{proof}
Letting $s=1/q+\alpha$, we observe that
\[
\alpha\ge 0\iff s\ge\frac{1}{q},\quad\eqref{qalpha}\iff s<\min\left\{\frac{2}{p},\frac{2}{p-1}-\frac{N-1}{q}\right\}.
\]
This together with Lemma~\ref{Bsqq} implies that (i) is equivalent to $\mu\in B^{-s}_{q,q}(\R^{N-1})$ for some $q>p$ and $s$ satisfying
\begin{equation}\label{pqs}
\frac{1}{q}\le s<\min\left\{\frac{2}{p},\frac{2}{p-1}-\frac{N-1}{q}\right\}.
\end{equation}
Since
\begin{equation}\label{pqs'}
\eqref{pqs}\implies \frac{1}{q}<\min\left\{\frac{2}{p},\frac{2}{p-1}-\frac{N-1}{q}\right\}\iff q>\max\left\{\frac{p}{2},\frac{N}{2}(p-1)\right\},
\end{equation}
(i) $\implies$ (ii) holds. Conversely, assume (ii). By \eqref{pqs'}, we find $s'\ge s$ such that
\[
\frac{1}{q}\le s'<\min\left\{\frac{2}{p},\frac{2}{p-1}-\frac{N-1}{q}\right\},
\]
and see that $\mu\in B^{-s'}_{q,q}(\R^{N-1})$. Thus (ii) $\implies$ (i) holds.
\end{proof}
\begin{Remark}\label{muassump}
\begin{enumerate}[label=\rm{(\roman*)}]
\item Assume that $1<p<p_S$. Then the embedding $H^{1/2}(\R^{N-1})\subset B^0_{q,q}(\R^{N-1})$ holds with $q=2(N-1)/(N-2)$. Since $(N-2)/2<2/(p-1)$ by $1<p<p_S$, the inequalities
\[
0<\frac{2}{p},\quad 0<\frac{2}{p-1}-\frac{N-1}{q}
\]
hold. Thus $\mu\in H^{1/2}(\R^{N-1})$ satisfies the assumption in Theorem~\ref{Thm1}.
\item Due to the embedding $C_0(\R^{N-1})\subset B^s_{r,r}(\R^{N-1})$ for $s>0$ and $r>1$ with $s>(N-1)/r$, every bounded Radon measure $\mu$~on~$\R^{N-1}$ belongs to $B^{-s}_{q,q}(\R^{N-1})$ for $s>0$ and $q>1$ with $s>(N-1)(1-1/q)$. Assume that
\[
1<p<\frac{N+1}{N-1}.
\]
Then the inequalities
\[
(N-1)\left(1-\frac{1}{p}\right)<\frac{2}{p}<\frac{2}{p-1}-\frac{N-1}{p},\quad p<\frac{N}{2}(p-1)
\]
hold. It follows that $q$ slightly larger than $p$ and $s$ slightly larger than $(N-1)(1-1/q)$ satisfy the condition \eqref{equiv}. Thus every bounded Radon measure $\mu$~on~$\R^{N-1}$ satisfies the assumption in Theorem~\ref{Thm1}.
\end{enumerate}
\end{Remark}
\def\thesection{Appendix \Alph{section}}
\section{Proofs of Propositions \ref{Glaa} and \ref{cpt}}\label{pLaG}
\def\thesection{\Alph{section}}
In this appendix, we prove the propositions given in Section~\ref{LaG}.
\subsection{Proof of Proposition~\ref{Glaa}}
We first prove Proposition~\ref{Glaa}, the integral inequality of the integral kernel $G$ in the spaces $L^q_\alpha$.
\begin{Lemma}\label{Gest}
\[
G(x,y)\le\min\left\{E(|x-y|),4x_Ny_N\frac{|E'(|x-y|)|}{|x-y|}\right\}\FA x,y\in\R^N_+\with x\neq y.
\]
\end{Lemma}
\begin{proof}
It is clear that $G(x,y)\le E(|x-y|)$. Since $|E'(r)|$ is decreasing in $r>0$, it follows from the mean value theorem that
\begin{align*}
G(x,y)\le |E'(|x-y|)|(|x-y|-|x^*-y|)&= |E'(|x-y|)|\cdot\frac{|x-y|^2-|x^*-y|^2}{|x-y|+|x^*-y|}\\
&\le 4x_Ny_N\frac{|E'(|x-y|)|}{|x-y|},
\end{align*}
and the proof of Lemma~\ref{Gest} is complete.
\end{proof}
\begin{Lemma}\label{Gintestt}
Let $s\in[1,N/(N-2))$ and $\theta\in\R$ be such that
\begin{equation}\label{thetacondition}
-1-\frac{1}{s}<\theta<N-1-\frac{N}{s}.
\end{equation}
Then
\begin{equation}\label{Gintest}
\left(\integ{\R^N_+}(G(x,y)h(y_N)^{\theta})^sdy\right)^{1/s}\le Ch(x_N)^{2+\theta-N(1-1/s)}\FA x\in\R^{N}_+.
\end{equation}
\end{Lemma}
\begin{proof}
We only prove \eqref{Gintest} in the case $N\ge 3$. The proof of \eqref{Gintest} in the case $N=1,2$ is similar.

We first prove \eqref{Gintest} in the case $x_N<1$. We set
\begin{align*}
I_1&:=
\integ{\{|x-y|<x_N/2\}}|x-y|^{(2-N)s}h(y_N)^{\theta s}dy,\\
I_2&:=\integ{\{x_N/2\le|x-y|<1/2,y_N<x_N\}}x_N^sy_N^{(1+\theta)s}|x-y|^{-Ns}dy,\\
I_3&:=\integ{\{|x-y|\ge 1/2,y_N<x_N\}}x_N^sy_N^{(1+\theta)s}e^{-s|x-y|/2}dy,\\
I_4&:=\integ{\{x_N/2\le|x-y|<1/2,y_N\ge x_N\}}x_N^sy_N^s|x-y|^{-Ns}h(y_N)^{\theta s}dy,\\
I_5&:=\integ{\{|x-y|\ge 1/2, x_N\le y_N<1\}}x_N^sy_N^{(1+\theta)s}e^{-s|x-y|/2}dy,\\
I_6&:=\integ{\{|x-y|\ge 1/2, y_N\ge 1\}}x_N^sy_N^se^{-s|x-y|/2}dy.
\end{align*}
Then, by Lemma~\ref{Gest} and the fact
\begin{equation}\label{Eord}
E(r)\le\begin{cases}
Cr^{2-N}&\jf r<1,\\
Ce^{-r/2}&\jf r\ge 1,
\end{cases}\quad
|E'(r)|\le\begin{cases}
Cr^{1-N}&\jf r<1,\\
Ce^{-r/2}&\jf r\ge 1,
\end{cases}
\end{equation}
we have
\begin{equation}\label{G1}
\integ{\R^N_+}(G(x,y)h(y_N)^\theta)^sdy\le C(I_1+I_2+I_3+I_4+I_5+I_6).
\end{equation}
We first estimate $I_1$. It follows from $x_N/2<y_N<3x_N/2$ for $y\in B(x,x_N/2)$, and $(N-2)s<N$ that
\begin{equation}\label{I1}
\begin{aligned}
I_1&\le C\integ{\{|x-y|<x_N/2\}}|x-y|^{(2-N)s}\min\{x_N,1\}^{\theta s}dy\\
&\le Cx_N^{(2-N)s+N}x_N^{\theta s}= Cx_N^{(2+\theta)s-N(s-1)}.
\end{aligned}
\end{equation}
We next estimate $I_2+I_3$. Since $e^{-t/2}\le Ct^{-N}$ for any $t\ge 1/2$ and $(1+\theta)s>-1$ by \eqref{thetacondition}, it follows that
\begin{equation}\label{I23}
\begin{aligned}
I_2+I_3&\le C\integ{\{|x-y|\ge x_N/2,y_N<x_N\}}x_N^sy_N^{(1+\theta)s}|x-y|^{-Ns}dy\\
&\le C\integ{\{y_N<x_N\}}x_N^sy_N^{(1+\theta)s}\max\left\{\frac{x_N}{2},|x-y|\right\}^{-Ns}dy\\
&\le C\integ0^{x_N}\left(\integ{\R^{N-1}}x_N^sy_N^{(1+\theta)s}\max\left\{\frac{x_N}{2},|x'-y'|\right\}^{-Ns}dy'\right)dy_N\\
&=C\integ0^{x_N}\left(\left(\frac{x_N}{2}\right)^{N-1}\integ{\R^{N-1}}\max\left\{\frac{x_N}{2},\frac{x_N}{2}|z'|\right\}^{-Ns}dz'\right)x_N^sy_N^{(1+\theta)s}dy_N\\
&=Cx_N^{N-1-Ns+s}x_N^{(1+\theta)s+1}=Cx_N^{(2+\theta)s-N(s-1)}.
\end{aligned}
\end{equation}
We estimate $I_4$. We observe from $y_N\le x_N+|x-y|< 3/2$ for $y\in B(x,1/2)$ that
\[
y_N\approx h(y_N)\for y\in B(x,1/2).
\]
Thus
\[
I_4\le C\integ{\{x_N/2\le|x-y|<1/2,y_N> x_N\}}x_N^sy_N^{(1+\theta)s}|x-y|^{-Ns}dy.
\]
If $\theta\le -1$, since $Ns-(1+\theta)s>N$ and $(1+\theta)s>-1$ by \eqref{thetacondition},
\begin{equation}\label{I4-1}
\begin{aligned}
I_4&\le C\integ{\{x_N/2\le|x-y|,y_N> x_N\}}x_N^s(y_N-x_N)^{(1+\theta)s}|x-y|^{-Ns}dy\\
&\le C\integ{\{|z|\ge x_N/2,z_N>0\}}x_N^sz_N^{(1+\theta)s}|z|^{-Ns}dz\\
&=Cx_N^s\integ{x_N/2}^\infty r^{(1+\theta)s-Ns+N-1}dr\cdot\integ{\{|\tau|=1,\tau_N>0\}}\tau_N^{(1+\theta)s}d\sigma(\tau)\le Cx_N^{(2+\theta)s-N(s-1)}.
\end{aligned}
\end{equation}
On the other hand, if $\theta>-1$, it follows from \eqref{thetacondition} that $N-1-N/s>-1$ and thus $s>1$. This together with $y_N\le x_N+|x-y|$ and $Ns-(1+\theta)s>N$ implies that
\begin{equation}\label{I4-2}
\begin{aligned}
I_4&\le C\integ{\{x_N/2\le|x-y|,y_N\ge x_N\}}x_N^s(x_N^{(1+\theta)s}+|x-y|^{(1+\theta)s})|x-y|^{-Ns}dy\\
&\le C\left(\integ{\{|x-y|\ge x_N/2\}}\left(x_N^{(2+\theta)s}|x-y|^{-Ns}+x_N^s|x-y|^{(1+\theta)s-Ns}\right)dy\right)\\
&\le C\left(x_N^{(2+\theta)s}x_N^{-N(s-1)}+x_N^sx_N^{(1+\theta)s-N(s-1)}\right)\le Cx_N^{(2+\theta)s-N(s-1)}.
\end{aligned}
\end{equation}
We estimate $I_5$ as
\begin{align*}
I_5&\le C\integ{\{|x-y|\ge 1/2,x_N\le y_N<1\}}x_N^s\max\{x_N^{(1+\theta)s},1\}e^{-s|x-y|/2}dy\\
&\le Cx_N^s\max\{x_N^{(1+\theta)s},1\}\integ{\{|x-y|\ge 1/2\}}e^{-s|x-y|/2}dy\\
&\le C\max\{x_N^{(2+\theta)s},x_N^s\}=Cx_N^{\min\{(2+\theta)s,s\}}.
\end{align*}
Since $(2+\theta)s>(2+\theta)s-N(s-1)$ and $s>(2+\theta)s-N(s-1)$ by \eqref{thetacondition},
\begin{equation}\label{I5}
I_5\le Cx_N^{(2+\theta)s-N(s-1)}.
\end{equation}
We finally estimate $I_6$. It follows from $y_N\le x_N+|x-y|$ and \eqref{thetacondition} that
\begin{equation}\label{I6}
\begin{aligned}
I_6&\le C\integ{\{|x-y|\ge 1/2,y_N\ge 1\}}x_N^s(x_N^s+|x-y|^s)e^{-s|x-y|/2}dy\\
&\le C\left(x_N^{2s}\integ{\{|x-y|\ge 1/2\}}e^{-s|x-y|/2}dy
+x_N^s\integ{\{|x-y|\ge 1/2\}}|x-y|^se^{-s|x-y|/2}dy\right)\\
&\le Cx_N^s\le Cx_N^{(2+\theta)s-N(s-1)}.
\end{aligned}
\end{equation}
Combining \eqref{G1}, \eqref{I1}, \eqref{I23}, \eqref{I4-1}, \eqref{I4-2}, \eqref{I5}, and \eqref{I6}, we obtain
\begin{align*}
\left(\integ{\R^N_+}(G(x,y)h(y_N)^\theta)^sdy\right)^{1/s}\le C x_N^{(2+\theta)s-N(s-1)}.
\end{align*}
We next prove \eqref{Gintest} in the case $x_N\ge 1$. It suffices to prove that
\[
\integ{\R^N_+}(G(x,y)h(y_N)^\theta)^sdy\le C.
\]
We set
\begin{align*}
I_7&:=\integ{\{|x-y|<1/2\}}|x-y|^{(2-N)s}h(y_N)^{\theta s}dy,\\
I_8&:=\integ{\{y_N<1/2,|x-y|\ge 1/2\}}x_N^sy_N^{(1+\theta)s}e^{-|x-y|/2}dy,\\
I_9&:=\integ{\{y_N\ge 1/2,|x-y|\ge 1/2\}}|x-y|^{(1-N)s/2}e^{-|x-y|s}h(y_N)^{\theta s}dy.
\end{align*}
Then it follows from Lemma~\ref{Gest} and \eqref{Eord} that
\[
\integ{\R^N_+}(G(x,y)h(y_N)^\theta)^sdy\le C(I_7+I_8+I_9).
\]
Since $1/2<h(y_N)\le 1$ for $y\in B(x,1/2)$ and $(N-2)s<N$, $I_7\le C$. It is clear that $I_9\le C$. It remains to estimate $I_8$. It follows from $x_N\le y_N+|x-y|$ and \eqref{thetacondition} that
\[
\begin{aligned}
I_8&\le C\integ{\{y_N<1/2\}}\left(y_N^{(2+\theta)s}+y_N^{(1+\theta)s}|x-y|^{s}\right)e^{-|x-y|s/2}dy\\
&\le C\integ{\{y_N<1/2\}}\left(y_N^{(2+\theta)s}+y_N^{(1+\theta)s}|x'-y'|^{s}\right)e^{-|x'-y'|s/2}dy\\
&\le C\integ0^{1/2}(y_N^{(2+\theta)s}+y_N^{(1+\theta)s})dy_N=C,
\end{aligned}
\]
and the proof of Lemma~\ref{Gintestt} is complete.
\end{proof}
\begin{proof}[Proof of Proposition~{\rm\ref{Glaa}}]
Let $\beta'\le\beta$. Set $1/s:=1-1/q+1/r$ and let $\theta_1$, $\theta_2\in\R$ be such that
\begin{equation}\label{theta2}
\left(1-\frac{1}{q}\right)s\theta_1+\frac{1}{r}s\theta_2+\left(1-\frac{1}{s}\right)q\alpha=0.
\end{equation}

Since $s\in[1,N/(N-2))$ by $1/r>1/q-2/N$, it follows from \eqref{theta2} and the H\"{o}lder inequality that
\begin{align*}
&\left|\integ{\R^N_+}G(x,y)f(y)dy\right|\\
&\le \integ{\R^N_+}\left(G(x,y)h(y_N)^{\theta_1}\right)^{s(1-1/q)}\left(G(x,y)^s|f(y)|^qh(y_N)^{s\theta_2}\right)^{1/r}\left(|f(y)|h(y_N)^\alpha\right)^{1-q/r}dy\\
&
\le\left(\integ{\R^N_+}\left(G(x,y)h(y_N)^{\theta_1}\right)^sdy\right)^{1-1/q}
\left(\integ{\R^N_+}\left(G(x,y)h(y_N)^{\theta_2}\right)^s|f(y)|^qdy\right)^{1/r}\|f\|_{L^q_\alpha}^{1-q/r}.
\end{align*}
This together with \eqref{Gintest} implies that
\begin{multline*}
\left|\integ{\R^N_+}G(x,y)f(y)dy\right|\\
\le C h(x_N)^{s(1-1/q)(2+\theta_1-N(1-1/s))}\|f\|_{L^q_\alpha}^{1-q/r}\left(\integ{\R^N_+}\left(G(x,y)h(y_N)^{\theta_2}\right)^s|f(y)|^qdy\right)^{1/r},
\end{multline*}
provided that
\begin{equation}\label{theta1-1}
-1-\frac{1}{s}<\theta_1<N-1-\frac{N}{s}.
\end{equation}
Under the condition \eqref{theta1-1},
\begin{align*}
&\integ{\R^N_+}\left|\integ{\R^N_+}G(x,y)f(y)dy\right|^rh(x_N)^{r\beta'}dx\\
&\le C\integ{\R^N_+}\left(\integ{\R^N_+}\left(G(x,y)h(y_N)^{\theta_2}\right)^s|f(y)|^qdy\right)h(x_N)^{rs(1-1/q)(2+\theta_1-N(1-1/s))+r\beta'}dx\cdot\|f\|_{L^q_\alpha}^{r-q}\\
&=C\integ{\R^N_+}\left(\integ{\R^N_+}\left(G(x,y)h(x_N)^{r(1-1/q)(2+\theta_1-N(1-1/s))+r\beta'/s}\right)^sdx\right)|f(y)|^qh(y_N)^{s\theta_2}dy\cdot\|f\|_{L^q_\alpha}^{r-q}.
\end{align*}
Since $G$ is symmetric, it follows from \eqref{Gintest} that
\[
\|G[f]\|_{L^r_{\beta'}}^r\le C\integ{\R^N_+}h(y_N)^{s(2+r(1-1/q)(2+\theta_1-N(1-1/s))+r\beta'/s-N(1-1/s))}\cdot|f(y)|^qh(y_N)^{s\theta_2}dy\cdot\|f\|_{L^q_\alpha}^{r-q},
\]
under the conditions \eqref{theta1-1} and
\begin{equation}\label{theta1-2}
-1-\frac{1}{s}<r\left(1-\frac{1}{q}\right)\left(2+\theta_1-N\left(1-\frac{1}{s}\right)\right)+\frac{r\beta'}{s}<N-1-\frac{N}{s}.
\end{equation}
Since
\begin{align*}
&s\left(2+r\left(1-\frac{1}{q}\right)\left(2+\theta_1-N\left(1-\frac{1}{s}\right)\right)+\frac{r\beta'}{s}-N\left(1-\frac{1}{s}\right)\right)+s\theta_2\\
&=r\left(\left(1-\frac{1}{q}\right)s\theta_1+\frac{1}{r}s\theta_2\right)+rs\left(1-\frac{1}{q}+\frac{1}{r}\right)\left(2-N\left(1-\frac{1}{s}\right)\right)+r\beta'\\
&=-r\left(1-\frac{1}{s}\right)q\alpha+r\left(2-N\left(1-\frac{1}{s}\right)\right)+r\beta'\\
&=-qr\left(1-\frac{1}{s}-\frac{1}{q}\right)\alpha+r\left(2-N\left(1-\frac{1}{s}\right)-\alpha+\beta'\right)=q\alpha+r\left(2-N\left(1-\frac{1}{s}\right)-\alpha+\beta'\right)
\end{align*}
by \eqref{theta2}, we obtain
\[
\|G[f]\|^r_{L^r_\beta}\le\|G[f]\|^r_{L^r_{\beta'}}\le  C\integ{\R^N_+}|f(y)|^qh(y_N)^{q\alpha}dy\cdot\|f\|_{L^q_\alpha}^{r-q}=C\|f\|_{L^q_\alpha}^r,
\]
if there are $\beta'\le\beta$ and $\theta_1\in\R$ satisfying \eqref{theta1-1}, \eqref{theta1-2}, and
\begin{equation}\label{beta'}
\beta'\ge\alpha-2+N\left(1-\frac{1}{s}\right).
\end{equation}
Since $r/s=1+r(1-1/q)>1$,
\begin{align*}
\eqref{theta1-2}&\iff -1-\frac{1}{s}-\frac{r\beta'}{s}<\left(\frac{r}{s}-1\right)\left(2+\theta_1-N\left(1-\frac{1}{s}\right)\right)<N-1-\frac{N}{s}-\frac{r\beta'}{s}\\
&\iff -\frac{s+1+r\beta'}{r-s}-2+N\left(1-\frac{1}{s}\right)<\theta_1<\frac{(N-1)s-N-r\beta'}{r-s}-2+N\left(1-\frac{1}{s}\right).
\end{align*}
Thus it suffices to find $\beta'\le\beta$ satisfying \eqref{beta'} and
\begin{equation}\label{beta'-1}
-1-\frac{1}{s}<\frac{(N-1)s-N-r\beta'}{r-s}-2+N\left(1-\frac{1}{s}\right),
\end{equation}
\begin{equation}\label{beta'-2}
-\frac{s+1+r\beta'}{r-s}-2+N\left(1-\frac{1}{s}\right)<N-1-\frac{N}{s},
\end{equation}
to prove the existence of $\beta'$ and $\theta_1$ satisfying \eqref{theta1-1}, \eqref{theta1-2}, and \eqref{beta'}. Furthermore, since
\[
\eqref{beta'-1}\iff \beta'<\frac{(N-1)s-N}{r}+\frac{(N-1)(r-s)}{r}\left(1-\frac{1}{s}\right)\iff \beta'<(N-1)\left(1-\frac{1}{s}\right)-\frac{1}{r}
\]
and
\[
\eqref{beta'-2}\iff \beta'>-\frac{s+1}{r}-\frac{r-s}{r}\iff \beta'>-1-\frac{1}{r},
\]
there is $\beta'\le\beta$ satisfying \eqref{beta'}, \eqref{beta'-1}, and \eqref{beta'-2}, if and only if $\beta$ satisfies
\begin{equation}\label{beta'-3}
\beta\ge\alpha-2+N\left(1-\frac{1}{s}\right),
\end{equation}
\[
\beta>-1-\frac{1}{r},
\]
\begin{equation}\label{beta'-6}
(N-1)\left(1-\frac{1}{s}\right)-\frac{1}{r}>\alpha-2+N\left(1-\frac{1}{s}\right),
\end{equation}
\begin{equation}\label{beta'-4}
(N-1)\left(1-\frac{1}{s}\right)-\frac{1}{r}>-1-\frac{1}{r}.
\end{equation}
\eqref{beta'-4} is true for any $q$, $r$ with $q<r$. It holds that
\begin{gather*}
\eqref{beta'-3}\iff \beta\ge\alpha-2+N\left(\frac{1}{q}-\frac{1}{r}\right)\iff\frac{N}{r}+\beta\ge\frac{N}{q}+\alpha-2,\\
\eqref{beta'-6}\iff \alpha+\left(1-\frac{1}{s}\right)+\frac{1}{r}<2\iff \frac{1}{q}+\alpha<2,
\end{gather*}
and the proof of Proposition~\ref{Glaa} is complete.
\end{proof}
\subsection{Proof of Proposition~\ref{cpt}}
We next prove Proposition~\ref{cpt}, the compactness of the integral operator $f\mapsto G[af]$.
\begin{proof}[Proof of Proposition~{\rm\ref{cpt}}]
Set
\[
\frac{1}{q}:=\frac{1}{r}+\frac{1}{s'},\quad \alpha:=\theta+\beta.
\]
It holds that $q\in(1,\infty)$ and
\[
\frac{1}{q}+\alpha<2,\quad\frac{1}{q}-\frac{1}{r}<\frac{2}{N},\quad \frac{N}{r}+\beta\ge\frac{N}{q}+\alpha-2,
\]
and hence $G[\cdot]$ is a bounded operator from $L^q_\alpha$ to $L^r_\beta$ by Proposition~\ref{Glaa}. This together with the H\"older inequality and \eqref{Cptassump} implies that
\begin{equation}\label{Taopnorm}
\|T_a\|_{L^r_\beta\to L^r_\beta}\le \|a\|_{L^{s'}_\theta}\|G[\cdot]\|_{L^q_\alpha\to L^r_\beta}\le C\|a\|_{L^{s'}_\theta}.
\end{equation}
We first consider the case $a\in C^\infty_c(\R^N_+)$. Let $q':=q/(q-1)$, $r':=r/(r-1)$, and $T_a':L^{r'}_{-\beta}\to L^{r'}_{-\beta}$ be the adjoint operator of $T_a$. It suffices to prove that $T_a'$ is compact. For any $f\in L^r_\beta$ and $g\in L^{r'}_{-\beta}$, it follows that
\begin{align*}
\integ{\R^N_+}f(T_a'g)dx=\integ{\R^N_+}(T_af)gdx&=\integ{\R^N_+}\left(\integ{\R^N_+}G(x,y)a(y)f(y)dy\right)g(x)dx\\
&=\integ{\R^N_+}\left(\integ{\R^N_+}G(x,y)g(x)dx\right)a(y)f(y)dy.
\end{align*}
Since $G$ is symmetric, it implies that
\begin{equation}\label{Ta'}
T_a'g=aG[g]\quad\textrm{for any}\quad g\in L^{r'}_{-\beta}.
\end{equation}
Let open sets $\Omega\Subset\Omega'\Subset\R^N_+$ be such that $\supp a\subset\Omega$. It follows from the $L^p$ estimates (see e.g. \cite[Theorem 9.11]{GT}) and \eqref{dual} that
\[
\|G[f]\|_{W^{2,r'}(\Omega)}\le C\left(\|f\|_{L^{r'}(\Omega')}+\|G[f]\|_{L^{q'}(\Omega')}\right)\le C\left(\|f\|_{L^{r'}_{-\beta}}+\|G[f]\|_{L^{q'}_{-\alpha}}\right)\le C\|f\|_{L^{r'}_{-\beta}}.
\]
Due to the compact embedding theorem, $f\mapsto G[f]$ is a compact operator from $L^{r'}_{-\beta}$ to $L^{r'}(\Omega)$. This together with
\[
\|ah\|_{L^{r'}_{-\beta}}\le C\|h\|_{L^{r'}(\Omega)}\FA h\in L^{r'}_\loc(\R^N_+)
\]
and \eqref{Ta'} implies that $T_a'$ is a compact operator from $L^{r'}_{-\beta}$ to itself.

We finally prove the general case. Let $\{a_k\}_{k\in\{1,2,\ldots\}}\subset C^\infty_c(\R^N_+)$ be such that $\|a-a_k\|_{L^{s'}_{\theta}}\to 0$ as $n\to\infty$. Then $T_{a_k}$ is a compact operator from $L^r_\beta$ to itself for each $k$. Since
\[
\|T_a-T_{a_k}\|_{L^r_\beta\to L^r_\beta}\le C\|a-a_k\|_{L^{s'}_\theta}\to 0\quad\textrm{as}\quad k\to\infty
\]
by \eqref{Taopnorm}, $T_a$ is a compact operator from $L^r_\beta$ to itself, and the proof of Proposition~\ref{cpt} is complete.
\end{proof}
\def\thesection{Appendix \Alph{section}}
\section{Proof of Lemma~{\rm\ref{Dstrategy}}}\label{Dst}
\def\thesection{\Alph{section}}
In this appendix, we give a proof of Lemma~\ref{Dstrategy}.  Instead of Lemma~\ref{Dstrategy}, we prove the following generalized version, which is needed for the proof of Lemma~\ref{psireg}.
\begin{Proposition}\label{Dstrategy'}
Assume the same conditions as in Theorem~{\rm\ref{Thm1}}. Let $r_0>1$ and $\beta\in\R$ be such that
\begin{equation}\label{s0b0}
\frac{1}{r_0}<1-\frac{p-1}{q},\quad\frac{1}{r_0}+\beta_0<2-(p-1)\left(\frac{1}{q}+\alpha\right).
\end{equation}
Then there is a sequence of sets $D_j(r_0,\beta_0)\subset (1,\infty)\times\R$, $j\in\{0,1,\ldots\}$, with the following properties.
\begin{enumerate}[label={\rm(\alph*)}]
\item $(r_0,\beta_0)\in D_0(r_0,\beta_0).$
\item For any $a\in L^{q/(p-1)}_{(p-1)\alpha}$ and $f\in \bigcap_{(r,\beta)\in D_j(r_0,\beta_0)} L^r_\beta$, $G[af]$ belongs to $\bigcap_{(r',\beta')\in D_j(r_0,\beta_0)}L^{r'}_{\beta'}$.
\item $D_j(r_0,\beta_0)=D_*$ for sufficiently large $j\in\{1,2,\ldots\}$.
\end{enumerate}
\end{Proposition}
Indeed, clearly from \eqref{qalpha}, the pair $(r_0,\beta_0)=(q,\alpha)$ satisfies the condition \eqref{s0b0}, and thus Proposition~\ref{Dstrategy'} implies Lemma~\ref{Dstrategy}.

We give a construction of $D_j(r_0,\beta_0)$. We set
\begin{equation}\label{tauchoice}
\tau:=\begin{dcases}
\min\left\{\frac{2}{N}-\frac{p-1}{q},\frac{1}{N-1}\left(2-(p-1)\left(\frac{N}{q}+\alpha\right)\right)\right\}&\jf N\ge 2,\\
2-\frac{p-1}{q}&\jf N=1,
\end{dcases}
\end{equation}
and
\begin{equation}\label{betaj}
\beta_j(r):=\max\left\{\beta_0+N\left(\frac{1}{r_0}-\frac{1}{r}\right)-j\left(2-(p-1)\left(\frac{N}{q}+\alpha\right)\right),-1-\frac{1}{r}\right\}
\end{equation}
for $j\in\{0,1,\ldots\}$ and $r>1$. Since
\[
2-(p-1)\left(\frac{N}{q}+\alpha\right)>0
\]
by \eqref{qalpha}, $\tau>0$ and $\beta_j(r)$ is nonincreasing in $j$ for any $r>1$. We define
\begin{equation}\label{D}
\begin{aligned}
&D_j(r_0,\beta_0)\\
&:=
\begin{dcases}
\{(r_0,\beta_0)\}&\for j=0,\\
\left\{(r,\beta): r>1,\frac{1}{r_0}-j\tau<\frac{1}{r}<\frac{1}{r_0}+j\frac{p-1}{q},\beta>\beta_j(r)\right\}&\for j\in\{1,2,\ldots\}.
\end{dcases}
\end{aligned}
\end{equation}
\begin{Lemma}\label{Diterat}
Assume the same conditions as in Theorem~{\rm\ref{Thm1}} and \eqref{s0b0}. Let $j\in\{0,1,\ldots\}$ and $(r',\beta')\in D_{j+1}(r_0,\beta_0)$. Then there is $(r,\beta)\in D_j(r_0,\beta_0)$ such that
\begin{gather}
\frac{1}{r'}\le\frac{1}{r}+\frac{p-1}{q}<\min\left\{\frac{1}{r'}+\frac{2}{N},1\right\},\label{Diterat1}\\
\frac{1}{r}+\beta<2-(p-1)\left(\frac{1}{q}+\alpha\right),\label{Diterat2}\\
\frac{N}{r}+\beta<\frac{N}{r'}+\beta'-(p-1)\left(\frac{N}{q}+\alpha\right)+2.\label{Diterat3}
\end{gather}
\end{Lemma}
\begin{proof}
We first prove the case $j=0$. It suffices to check \eqref{Diterat1}, \eqref{Diterat2}, and \eqref{Diterat3} with $(r,\beta)=(r_0,\beta_0)$ and $\beta'>\beta_1(r')$. \eqref{Diterat2} is included in assumption \eqref{s0b0}. Furthermore, it follows from $(r',\beta')\in D_1(r_0,\beta_0)$ and \eqref{s0b0} that
\[
\frac{1}{r'}<\frac{1}{r_0}+\frac{p-1}{q}<\min\left\{\frac{1}{r_0}+\frac{2}{N},1\right\}.
\]
Furthermore, by $(r',\beta')\in D_1(r_0,\beta_0)$ and \eqref{betaj},
\begin{align*}
&\left(\frac{N}{r_0}+\beta_0\right)-\left(\frac{N}{r'}+\beta'\right)<\beta_0+N\left(\frac{1}{r_0}-\frac{1}{r'}\right)-\beta_1(r')\\
&\le\beta_0+N\left(\frac{1}{r_0}-\frac{1}{r'}\right)-\left(\beta_0+N\left(\frac{1}{r_0}-\frac{1}{r'}\right)-\left(2-(p-1)\left(\frac{N}{q}+\alpha\right)\right)\right)\\
&=2-(p-1)\left(\frac{N}{q}+\alpha\right).
\end{align*}
These complete the case $j=0$.

We next prove the case $j\in\{1,2,\ldots\}$. Let $(r',\beta')\in D_{j+1}(r_0,\beta_0)$. We divide into two cases.
\begin{enumerate}[label=(\Roman*)]
\item \begin{equation}\label{itercase1}
\frac{1}{r_0}-j\tau+\frac{p-1}{q}<\frac{1}{r'}<\frac{1}{r_0}+(j+1)\frac{p-1}{q};
\end{equation}
\item \begin{equation}\label{itercase2}
\frac{1}{r_0}-(j+1)\tau<\frac{1}{r'}\le\frac{1}{r_0}-j\tau+\frac{p-1}{q}.
\end{equation}
\end{enumerate}
We first prove the case (I). We set
\[
\frac{1}{r}=\frac{1}{r'}-\frac{p-1}{q},
\]
and observe from \eqref{itercase1} that
\begin{equation}\label{itercase1-0}
\frac{1}{r_0}-j\tau<\frac{1}{r}<\frac{1}{r_0}+j\frac{p-1}{q}.
\end{equation}
Thus $(r,\beta)\in D_j(r_0,\beta_0)$ if $\beta>\beta_j(r)$. We find $\beta\in\R$ slightly larger than $\beta_j(r)$ such that $(r,\beta)$ satisfies \eqref{Diterat1}, \eqref{Diterat2}, and \eqref{Diterat3}. To this end, it suffices to check \eqref{Diterat1} and
\begin{gather}
\frac{1}{r}+\beta_j(r)<2-(p-1)\left(\frac{1}{q}+\alpha\right),\label{itercase1-2}\\
\frac{N}{r}+\beta_j(r)<\frac{N}{r'}+\beta'-(p-1)\left(\frac{N}{q}+\alpha\right)+2.\label{itercase1-3}
\end{gather}
\eqref{Diterat1} follows from the choice of $r$. Indeed,
\[
\frac{1}{r}+\frac{p-1}{q}=\frac{1}{r'}<\min\left\{\frac{1}{r'}+\frac{2}{N},1\right\}.
\]
We next check \eqref{itercase1-2}. By \eqref{betaj},
\[
\frac{1}{r}+\beta_j(r)=\max\left\{\frac{1}{r}+\beta_0+N\left(\frac{1}{r_0}-\frac{1}{r}\right)-j\left(2-(p-1)\left(\frac{N}{q}+\alpha\right)\right),-1\right\}.
\]
It follows from \eqref{s0b0}, \eqref{tauchoice}, and \eqref{itercase1-0} that
\begin{align*}
&\frac{1}{r}+\beta_0+N\left(\frac{1}{r_0}-\frac{1}{r}\right)-j\left(2-(p-1)\left(\frac{N}{q}+\alpha\right)\right)\\
&=\frac{1}{r_0}+\beta_0+(N-1)\left(\frac{1}{r_0}-\frac{1}{r}\right)-j\left(2-(p-1)\left(\frac{N}{q}+\alpha\right)\right)\\
&<\frac{1}{r_0}+\beta_0+(N-1)j\tau-j\left(2-(p-1)\left(\frac{N}{q}+\alpha\right)\right)\\
&\le\frac{1}{r_0}+\beta_0+j\left(2-(p-1)\left(\frac{N}{q}-\alpha\right)\right)-j\left(2-(p-1)\left(\frac{N}{q}+\alpha\right)\right)\\
&=\frac{1}{r_0}+\beta_0<2-(p-1)\left(\frac{1}{q}+\alpha\right).
\end{align*}
Furthermore, by \eqref{qalpha},
\[
-1<-1+(p-1)\left(\frac{2}{p}-\left(\frac{1}{q}+\alpha\right)\right)<2-(p-1)\left(\frac{1}{q}+\alpha\right).
\]
Thus \eqref{itercase1-2} holds.

We finally check \eqref{itercase1-3}. It follows from $\beta'>\beta_{j+1}(r')$, \eqref{qalpha}, \eqref{betaj}, and the inequality
\begin{equation}\label{abcd}
\begin{aligned}
\max\{a,b\}-\max\{c,d\}&=\max\{\min\{a-c,a-d\},\min\{b-c,b-d\}\}\\
&\le\max\{a-c,b-d\}\FA a,b,c,d\in\R
\end{aligned}
\end{equation}
that
\begin{align*}
&\left(\frac{N}{r}+\beta_j(r)\right)-\left(\frac{N}{r'}+\beta'\right)<-N\left(\frac{1}{r'}-\frac{1}{r}\right)+\beta_j(r)-\beta_{j+1}(r')\\
&\begin{multlined}
=-N\left(\frac{1}{r'}-\frac{1}{r}\right)+\max\left\{\beta_0+N\left(\frac{1}{r_0}-\frac{1}{r}\right)-j\left(2-(p-1)\left(\frac{N}{q}+\alpha\right)\right),-1-\frac{1}{r}\right\}\\
-\max\left\{\beta_0+N\left(\frac{1}{r_0}-\frac{1}{r'}\right)-(j+1)\left(2-(p-1)\left(\frac{N}{q}+\alpha\right)\right),-1-\frac{1}{r'}\right\}
\end{multlined}\\
&\le-N\left(\frac{1}{r'}-\frac{1}{r}\right)+\max\left\{N\left(\frac{1}{r'}-\frac{1}{r}\right)+\left(2-(p-1)\left(\frac{N}{q}+\alpha\right)\right),\frac{1}{r'}-\frac{1}{r}\right\}\\
&=-N\left(\frac{1}{r'}-\frac{1}{r}\right)+N\left(\frac{1}{r'}-\frac{1}{r}\right)+\left(2-(p-1)\left(\frac{N}{q}+\alpha\right)\right)=2-(p-1)\left(\frac{N}{q}+\alpha\right),
\end{align*}
and the proof of the case (I) is complete.

We next prove the case (II).
Let
\[
\frac{1}{r_j}:=\max\left\{\frac{1}{r_0}-j\tau,0\right\}.
\]
$(r,\beta)$ belongs to $D_j(r_0,\beta_0)$, for $r$ slightly less than $r_j$ and $\beta>\beta_j(r)$. We find $(r,\beta)$, with $r$ slightly less than $r_j$, and $\beta$ slightly larger than $\beta_j(r)$, such that $(r,\beta)$ satisfies \eqref{Diterat1}, \eqref{Diterat2}, and \eqref{Diterat3}. To this end, it suffices to check that
\begin{gather}
\frac{1}{r'}\le\frac{1}{r_j}+\frac{p-1}{q}<\min\left\{\frac{1}{r'}+\frac{2}{N},1\right\}\label{itercase2-1},\\
\frac{1}{r_j}+\beta_j(r_j)<2-(p-1)\left(\frac{1}{q}+\alpha\right)\label{itercase2-2},\\
\left(\frac{N}{r_j}+\beta_j(r_j)\right)-\left(\frac{N}{r'}+\beta'\right)<2-(p-1)\left(\frac{N}{q}+\alpha\right).\label{itercase2-3}
\end{gather}
We first check \eqref{itercase2-1}. It follows from \eqref{itercase2} that
\[
\frac{1}{r'}\le\frac{1}{r_0}-j\tau+\frac{p-1}{q}\le\frac{1}{r_j}+\frac{p-1}{q}.
\]
Furthermore, it follows from \eqref{s0b0}, \eqref{tauchoice}, and \eqref{itercase2} that
\begin{align*}
&\frac{1}{r_0}-j\tau+\frac{p-1}{q}<\min\left\{\left(\frac{1}{r'}+(j+1)\tau\right)-j\tau+\frac{p-1}{q},\frac{1}{r_0}+\frac{p-1}{q}\right\}\\
&\le\min\left\{\frac{1}{r'}+\tau+\frac{p-1}{q},1\right\}\le\min\left\{\frac{1}{r'}+\left(\frac{2}{N}-\frac{p-1}{q}\right)+\frac{p-1}{q},1\right\}=\min\left\{\frac{1}{r'}+\frac{2}{N},1\right\}.
\end{align*}
These together with $(p-1)/q<\min\left\{2/N,1\right\}$ imply that
\[
\frac{1}{r'}\le\frac{1}{r_j}+\frac{p-1}{q}=\max\left\{\frac{1}{r_0}-j\tau+\frac{p-1}{q},\frac{p-1}{q}\right\}<\min\left\{\frac{1}{r'}+\frac{2}{N},1\right\}.
\]

Since $1/r_0-1/r_j\le j\tau$, \eqref{itercase2-2} can be proved by the same argument as in the proof of \eqref{itercase1-2}.

We finally check \eqref{itercase2-3}. It follows from $(r',\beta')\in D_{j+1}(r_0,\beta_0)$, \eqref{tauchoice}, \eqref{betaj}, and \eqref{abcd} that
\begin{align*}
&\left(\frac{N}{r_j}+\beta_j(r_j)\right)-\left(\frac{N}{r'}+\beta'\right)<\left(\frac{N}{r_j}+\beta_j(r_j)\right)-\left(\frac{N}{r'}+\beta_{j+1}(r')\right)\\
&\begin{multlined}
=\max\left\{\frac{N}{r_j}+\beta_0+N\left(\frac{1}{r_0}-\frac{1}{r_j}\right)-j\left(2-(p-1)\left(\frac{N}{q}+\alpha\right)\right),\frac{N}{r_j}-1-\frac{1}{r_j}\right\}\\
-\max\left\{\frac{N}{r'}+\beta_0+N\left(\frac{1}{r_0}-\frac{1}{r'}\right)-(j+1)\left(2-(p-1)\left(\frac{N}{q}+\alpha\right)\right),\frac{N}{r'}-1-\frac{1}{r'}\right\}
\end{multlined}\\
&\le\max\left\{2-(p-1)\left(\frac{N}{q}+\alpha\right),(N-1)\left(\frac{1}{r_j}-\frac{1}{r'}\right)\right\}\\
&=\max\left\{2-(p-1)\left(\frac{N}{q}+\alpha\right),(N-1)\max\left\{\frac{1}{r_0}-j\tau-\frac{1}{r'},-\frac{1}{r'}\right\}\right\}\\
&\le\max\left\{2-(p-1)\left(\frac{N}{q}+\alpha\right),(N-1)\max\left\{\frac{1}{r_0}-j\tau-\left(\frac{1}{r_0}-(j+1)\tau\right),0\right\}\right\}\\
&=\max\left\{2-(p-1)\left(\frac{N}{q}+\alpha\right),(N-1)\tau,0\right\}=2-(p-1)\left(\frac{N}{q}+\alpha\right),
\end{align*}
and the proof of Lemma~\ref{Diterat} is complete.
\end{proof}

\begin{proof}[Proof of Proposition~{\rm\ref{Dstrategy'}}]
It follows from \eqref{D} that $\{D_j(r_0,\beta_0)\}_{j\in\{0,1,\ldots\}}$ satisfies (a) and (c). To check (b), let $(r',\beta')\in D_{j+1}(r_0,\beta_0)$ and $(r,\beta)\in D_j(r_0,\beta_0)$ be as in Lemma~\ref{Diterat}. Set
\[
\frac{1}{\overline{r}}:=\frac{1}{r}+\frac{p-1}{q},\quad \overline{\beta}:=\beta+(p-1)\alpha.
\]
It follows from \eqref{Diterat1}, \eqref{Diterat2}, and \eqref{Diterat3} that $1<\overline{r}<r'$ and
\[
\frac{1}{\overline{r}}+\overline{\beta}<2,\quad\frac{1}{\overline{r}}-\frac{1}{r'}<\frac{2}{N},\quad\frac{N}{r'}+\beta'>\frac{N}{\overline{r}}+\overline{\beta}-2.
\]
Furthermore, $(r',\beta')\in D_{j+1}(r_0,\beta_0)$ implies that $1/r'+\beta'>-1$. We deduce from Proposition~\ref{Glaa} and the H\"{o}lder inequality that $G[af]\in L^{r'}_{\beta'}$ for any $f\in L^{r}_{\beta}$.
This implies (b) and completes the proof of Proposition~\ref{Dstrategy'}.
\end{proof}
\medskip

\noindent
{\bf Acknowledgment.}
The author was supported in part by FoPM, WINGS Program, the University of Tokyo. We also thank Professor Toru Kan for helpful discussions and comments.

\end{document}